\documentclass[onecolumn,11pt]{article}

\usepackage{algorithmic}
\usepackage[ruled,vlined]{algorithm2e}
\usepackage{xcolor}
\usepackage{caption}
\usepackage{amsmath,amsfonts}
\usepackage{multirow}
\usepackage{graphicx}
\usepackage{breqn}
\usepackage{subcaption}
\usepackage{float}
\usepackage{enumerate}
\usepackage{booktabs}
\usepackage[a4paper,margin=1in]{geometry}
\usepackage[onehalfspacing]{setspace}
\usepackage[utf8]{inputenc}
\usepackage{enumitem}
\usepackage{verbatim}
\usepackage{tabularx}
\usepackage{makecell}
\usepackage{bbm}
\usepackage[hyphens]{url}
\usepackage{enumerate}
\usepackage{authblk}
\usepackage{adjustbox}
\usepackage{changepage}
\usepackage[short,nocomma]{optidef}
\usepackage[numbers]{natbib}

\usepackage[colorinlistoftodos]{todonotes}

\usepackage{soul}
\definecolor{dr}{rgb}{0.75,0.00,0.00}
\definecolor{lr}{rgb}{1.00,0.75,0.75}
\sethlcolor{lr}

\usepackage[colorlinks = true,
linkcolor = blue,
urlcolor  = blue,
citecolor = black,
anchorcolor = blue]{hyperref}

\usepackage{amsthm}
\newtheorem{myprop}{Proposition}
\newtheorem*{myprop*}{Proposition}

\title{Optimizing Autonomous Transfer Hub Networks:\protect\\Quantifying the Potential Impact of Self-Driving Trucks}
\author[1]{Chungjae~Lee}
\author[1]{Kevin~Dalmeijer}
\author[1]{Pascal~Van~Hentenryck}
\author[2,1]{Peibo~Zhang}
\affil[1]{H. Milton Stewart School of Industrial and Systems Engineering,\protect\\ Georgia Institute of Technology}
\affil[2]{Goizueta Business School, Emory University}
\date{\today}

\begin{document}

\maketitle

\pagenumbering{arabic}
\vfill
\begin{abstract}
	\noindent Autonomous trucks are expected to fundamentally transform the freight transportation industry.
	In particular, Autonomous Transfer Hub Networks (ATHNs), which combine autonomous trucks on middle miles with human-driven trucks on the first and last miles, are seen as the most likely deployment pathway for this technology. 
	This paper presents a framework to optimize ATHN operations and evaluate the benefits of autonomous trucking.
	By exploiting the problem structure, this paper introduces a flow-based optimization model for this purpose that can be solved by blackbox solvers in a matter of hours.
	The resulting framework is easy to apply and enables the data-driven analysis of large-scale systems.
	The power of this approach is demonstrated on a system that spans all of the United States over a four-week horizon.
	The case study quantifies the potential impact of autonomous trucking and shows that ATHNs can have significant benefits over traditional transportation networks.
\end{abstract}
\clearpage

\section{Introduction}
Self-driving trucks are expected to fundamentally transform the freight transportation industry.
Morgan Stanley estimates the potential savings from self-driving trucks at \$168 billion annually for the United States alone \citep{Greene2013-AutonomousFreightVehicles}.
Additionally, autonomous transportation may improve on-road safety, and reduce emissions and traffic congestion \citep{ShortMurray2016-IdentifyingAutonomousVehicle,SlowikSharpe2018-AutomationLongHaul}.

\citeauthor{SAEInternational2018-TaxonomyDefinitionsTerms} defines different levels of driving automation, ranging from L0 to L5, corresponding to no-driving automation to full-driving automation \citep{SAEInternational2018-TaxonomyDefinitionsTerms}.
The current focus is on L4 technology (high automation), which aims at delivering automated trucks that can drive without any human intervention in specific domains, e.g., on highways.
The trucking industry is actively involved in making L4 vehicles a reality.
Daimler Trucks, one of the leading heavy-duty truck manufacturers in North America, acquired a majority stake in self-driving truck developer Torc Robotics, which laid out a roadmap to launch autonomous trucks in 2027 \citep{TransportTopics2023-TorcLaysOut}.
Autonomous trucking company TuSimple has recently completed the first driverless tests on Chinese public roads \citep{TechCrunch2023-TusimpleTestsRemoving}.
In the US, Aurora Innovation teamed up with FedEx to haul freight between Fort Worth and El Paso, Texas, and the company reports that 60,000 miles have been completed without incidents \citep{FedEx2022-FedexAuroraExpand}.
These are just some of the companies involved in autonomous trucking, and others include Embark, Gatik, Kodiak, and Plus \citep{FleetOwner2021-TusimpleAutonomousTruck,Forbes2021-PlusPartnersIveco,FreightWaves2021-GatikIsuzuPartner}.

A study by \citeauthor{Viscelli-Driverless?AutonomousTrucks} describes different scenarios for the adoption of autonomous trucks by the industry \citep{Viscelli-Driverless?AutonomousTrucks}.
The most likely scenario, according to some of the major players, is the \emph{transfer hub business model} ~\citep{Viscelli-Driverless?AutonomousTrucks,RolandBerger2018-ShiftingGearAutomation,ShahandashtEtAl2019-AutonomousVehiclesFreight}.
Joanna Buttler, head of Daimler's global autonomous technology group, for example, stated that ``We are staying laser focused on U.S. hub-to-hub, on-highway'' \citep{TransportTopics2023-TorcLaysOut}.
An Autonomous Transfer Hub Network (ATHN) makes use of autonomous truck ports, or \emph{transfer hubs}, to hand off trailers between human-driven trucks and driverless autonomous trucks.
Autonomous trucks then carry out the transportation between the hubs, while regular trucks serve the first and last miles (see Figure~\ref{fig:autonomous_example}). 
Orders are split into a first-mile leg, an autonomous leg, and a last-mile leg, each of which served by a different vehicle.
A human-driven truck picks up the freight at the customer location, and drops it off at a nearby transfer hub.
A driverless self-driving truck moves the trailer to a transfer hub close to the destination, and another human-driven truck performs the last leg.

\begin{figure}[!t]
	\centering
	\includegraphics[width=\linewidth]{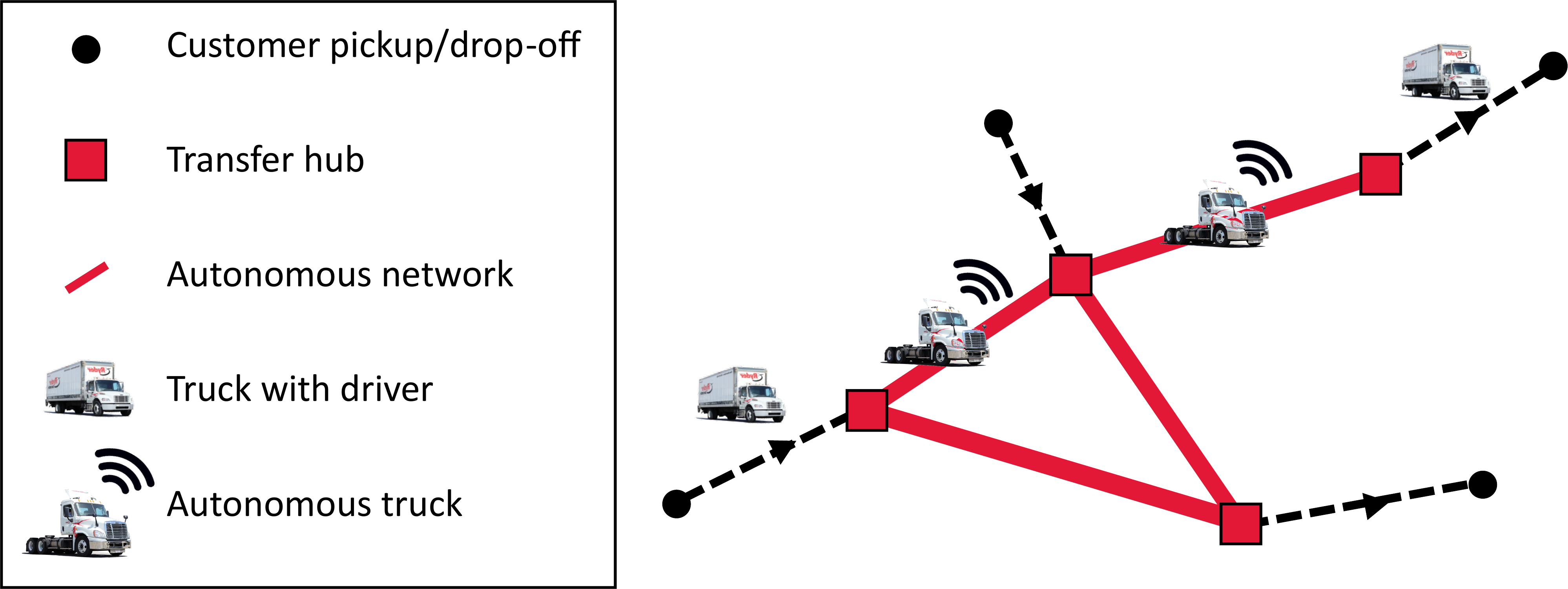}
	\caption{Example of an Autonomous Transfer Hub Network.}
	\label{fig:autonomous_example}
\end{figure}

The ATHN applies automation where it counts: Monotonous highway driving is automated, while more complex local driving and customer contact is left to humans.
Global consultancy firm \citeauthor{RolandBerger2018-ShiftingGearAutomation} estimates operational cost savings between 22\% and 40\% in the transfer hub model, based on cost estimates for three example trips \citep{RolandBerger2018-ShiftingGearAutomation}.
A recent white paper published by Ryder System, Inc. and the Socially Aware Mobility Lab studies whether these savings can be attained for actual operations and realistic orders in Full Truckload (FTL) shipping \citep{RyderSAM2021-ImpactAutonomousTrucking}.
It models ATHN operations as a scheduling problem and uses a Constraint Programming (CP) model to minimize empty miles and produce savings from 27\% to 40\% on a case study in the Southeast of the United States.

The current paper is an extension of the Ryder white paper that substantially improves, simplifies, and generalizes the methodology.
It is the culmination of two years of research into the core computational difficulty of optimizing ATHN operations, reported in the Ryder white paper and in technical reports by the authors \citep{DalmeijerVanHentenryck2021-OptimizingFreightOperations,LeeEtAl2022-OptimizationModelsAutonomous}.
The CP model presented in \citep{DalmeijerVanHentenryck2021-OptimizingFreightOperations} produces solutions that outperform the current operations, but that do not provide a bound on optimality.
\citep{LeeEtAl2022-OptimizationModelsAutonomous} introduces a Column Generation (CG) approach and a bespoke Network Flow (NF) model.
It is shown that the CP solution can be more than 10\% from optimal, and that the NF model can quickly produce solutions within 1\% from optimality.
These earlier findings motivate the flow-based optimization model in this paper that exploits the problem structure and is solved to optimality by blackbox solvers in a matter of hours.
The resulting framework is easy to apply and enables the data-driven analysis of large-scale systems.
It has also enabled a follow-up study on the role of hub capacities in ATHNs \citep{LeeEtAl2023-ConstraintProgrammingImprove}.

The power of the new methodology is demonstrated on an FTL system that spans all of the United States over a four-week horizon, expanding both the region and time horizon used in earlier reports.
The case study quantifies the potential impact of self-driving trucks and shows that ATHNs yield significant benefit over traditional transportation networks.
The main contributions of this work can be summarized as follows:
\begin{enumerate}
	\item The paper provides a high-level framework to optimize ATHN operations.
	\item The paper demonstrates that this enables the study of large-scale systems, requiring only a blackbox solver.
	\item The paper uses realistic order data to quantify the potential impact of FTL autonomous trucking in the US on a national scale.
\end{enumerate}

\noindent The remainder of this paper is organized as follows.
Section~\ref{sec:litreview} presents an overview of the literature.
Section~\ref{sec:problem} provides the problem description and Section~\ref{sec:methodology} discusses the methodology for optimizing ATHNs.
This methodology is applied to a case study in the US that is introduced in Section~\ref{sec:casestudy}.
The baseline results and the analysis of the potential impact of autonomous trucking are presented in Section~\ref{sec:baseline} and a detailed sensitivity analysis is provided by Section~\ref{sec:results}.
Finally, Section~\ref{sec:conclusion} provides the conclusions.

\section{Literature Review}
\label{sec:litreview}

As autonomous technology advances, more papers are studying the effect of autonomous vehicles on transportation systems.
\citep{Flaemig2016-AutonomousDrivingTechnical} provides an overview of the different ways that autonomous vehicles can be used both on public infrastructure and on private property (e.g., warehouses or company grounds).
In the urban transportation setting, \citep{CorreiaArem2016-SolvingUserOptimum} studies the effect of autonomous vehicles on traffic delays and parking demand in a city.
The authors use convex optimization to determine traffic assignments and a mixed integer nonlinear formulation to assign autonomous vehicles to households.
A case study for the city of Delft, The Netherlands, demonstrates a positive impact on the road network.

In the freight transportation context, routing and scheduling problems with autonomous trucks have gained attention very recently.
\citep{ChenEtAl2021-AutonomousTruckScheduling} considers scheduling a platoon of autonomous trucks to reduce air resistance when traveling between two seaport terminals in Singapore.
The authors present a mixed integer second-order-cone formulation that is solved with a column-generation based heuristic.
In the area of service network design, \citep{ScherrEtAl2018-ServiceNetworkDesign} proposes a problem where a human-driven truck leads a platoon of autonomous vehicles in the first tier of city logistics.
An arc-based mixed integer programming model on a time-space network is presented, but empirical observations show that only small problem instances are tractable.
\citep{ScherrEtAl2020-DynamicDiscretizationDiscovery} extends this work by introducing a dynamic discretization discovery approach that outperforms a commercial solver, and also present a heuristic to quickly generate solutions.

In the Less-Than-Truckload (LTL) context, \citep{AlHajjHassanEtAl2022-DailyLoadPlanning} studies the daily load planning problem under different levels of automation.
The paper focuses on modifying a given base plan to deal with dynamic load requests and other aspects that are important during operations, including driver regulations where drivers are involved.
The authors present a column-generation based heuristic to solve industry-based instances with up to 20 hubs and 1500 loads over a one-week horizon.

In terms of the problem structure, optimizing full truckload ATHN operations can be seen as a Pickup and Delivery Problem with Time Windows (PDPTW), where trucks pick up and drop off loads within the time windows prescribed by the customers.
The book \citet{TothVigo2014-VehicleRoutingProblems} provides a survey of this vehicle routing problem and other variants.
However, instead of routing, this paper will exploit the problem structure and take the perspective of scheduling a sequence of tasks (combined pickups and deliveries), which is closely related to the Vehicle Scheduling Problem with Time Windows (VSPTW, \citep{DesrosiersEtAl1995-TimeConstrainedRouting}).
These problems are well studied, and several exact and heuristic solution methods exist.
For example, \citep{FrelingEtAl2001-ModelsAlgorithmsSingle} presents a solution method based on the primal-dual algorithm framework for VSPTW with a single depot.
\citep{RibeiroSoumis1994-ColumnGenerationApproach} proposes a column-generation approach for the VPSTW with multiple depots, and \citep{HadjarEtAl2006-BranchCutAlgorithm} presents a branch-and-cut algorithm for the same problem.
\citep{SteinzenEtAl2010-TimeSpaceNetwork} considers solving the time-extended variant of the VSPTW with multiple depots using a heuristic based on the branch-and-price framework.
\citep{CampbellSavelsbergh} presents insertion heuristics for vehicle routing and scheduling problems.

The Vehicle Routing Problem with Full Truckloads (VRPFL, \citep{ArunapuramEtAl2003-VehicleRoutingScheduling}) is the specific variant that perhaps most structurally resembles the ATHN problem.
Similar to the current paper, the VRPFL asks for minimum-cost truck routes to serve a set of loads that are specified by an origin, destination, and a pickup time window.
\citep{ArunapuramEtAl2003-VehicleRoutingScheduling} proposes a branch-and-price framework as the solution approach.
The authors assume that each order consumes the full capacity of the truck, and the same assumption is made for optimizing ATHN operations, which reflects that autonomous trucks are expected to be mostly used for long-haul trips.
A crucial technical difference between \citep{ArunapuramEtAl2003-VehicleRoutingScheduling} and the current paper is that autonomous trucks are assumed to be completely interchangeable.
This will allow for a flow-based optimization model that is amenable to blackbox solving.

This paper introduces a high-level framework to optimize ATHN operations.
The goal of this framework is to provide a practical way to study large-scale autonomous FTL systems and to quantify the potential impact of autonomous trucking.
Previous works often rely on advanced optimization techniques such as cutting planes or column generation, or provide methods that do not scale to industry-sized problems.
For example, the largest problem considered by \citep{ArunapuramEtAl2003-VehicleRoutingScheduling} involves only 5 hubs and 160 loads.
In contrast, this paper exploits the problem structure to provide a model that is blackbox solvable on a large scale (up to 200 hubs and 6000+ loads over a four-week horizon).
Another benefit of the high-level framework is that it can be used to generate a base plan that forms the basis for the operational decisions, e.g., as studied by \citep{AlHajjHassanEtAl2022-DailyLoadPlanning} for LTL trucking.

\section{Problem Description}
\label{sec:problem}

This section introduces the problem of optimizing ATHN operations, while the solution methodology is presented in Section~\ref{sec:methodology}.
Table~\ref{tab:problem_symbols} summarizes the nomenclature for the problem description.
The goal is to serve a set of $n$ full truckloads $L$ at minimum cost with a combination of deliveries through the autonomous network and direct deliveries with regular trucks.
Each load $l \in L$ is identified by an origin location $o(l)$, a destination location $d(l)$, and a planned departure time, or release time, $r(l)$.
The autonomous network is based on a set of transfer hubs $V_H$.
Every load $l \in L$ is associated with an origin hub $h^+_l\in V_H$ near the origin $o(l)$ and a destination hub $h^-_l \in V_H$ near the destination $d(l)$.

\begin{figure*}[!tp]
	\scriptsize
	\centering
	\begin{tabular}{p{0.05\textwidth}p{0.90\textwidth}}
		\toprule
		Symbol & Definition \\
		\midrule
		\multicolumn{2}{l}{\textbf{Sets and Graphs}} \\
		$L$ & set of loads, each load $l \in L$ consists of an origin $o(l)$, a destination $d(l)$ and a release time $r(l)$.\\
		$V_H$ & set of autonomous transfer hub locations, $V_H \subseteq V$.\\
		$G$ & $=(V,A)$, location graph that models locations and connections in the ATHN.\\
		$V$ & set of locations.\\
		$A$ & set of location arcs, each arc $(i,j) \in A$ corresponds to travel from location $i\in V$ to location $j \in V$.\\
		\multicolumn{2}{l}{\textbf{Parameters}} \\
		$n$ & number of loads, $n=\lvert L \rvert$.\\
		$h_l^+$ & origin hub for load $l \in L$, $h_l^+ \in V_H$.\\
		$h_l^-$ & destination hub for load $l \in L$, $h_l^- \in V_H$.\\
		$K$ & maximum number of autonomous trucks.\\
		$\Delta$ & flexibility around the planned departure time (depart up to $\Delta$ earlier or later than planned).\\
		$S$ & autonomous truck loading/unloading time.\\
		$c_a$ & distance to travel location arc $a \in A$, $c_a \ge 0$.\\
		$\tau_a$ & time to travel location arc $a \in A$, $\tau_a > 0$.\\
		$\alpha$ & discount factor for autonomous mileage, $\alpha \in [0,1]$\\
		$\beta$ & first/last-mile inefficiency, $\beta \in [0,1)$\\
		\bottomrule
	\end{tabular}
	\captionof{table}{Nomenclature Problem Description.}%
	\label{tab:problem_symbols}%
\end{figure*}

\paragraph{Solution}
A solution consists of three types of decisions that are made jointly.
First, it is determined how each load $l \in L$ is served.
It is assumed that there are exactly two options:
\begin{itemize}
	\item \textbf{Autonomous:} The load follows the path $o(l) \rightarrow h^+_l \rightarrow h^-_l \rightarrow d(l)$. The first and last legs are performed by a regular truck, while the connection between the hubs is served by an autonomous truck.
	\item \textbf{Direct:} The load follows the path $o(l) \rightarrow d(l) \rightarrow o(l)$. Both legs are served by a single regular truck that returns empty.
	Note that the case study will consider challenging orders that actually incur such an empty return in practice.
\end{itemize}
Second, the autonomous legs ($h^+_l \rightarrow h^-_l$) of the loads that are served autonomously are combined into \emph{routes} for at most $K \ge 0$ autonomous trucks.
Note that these routes may include empty relocations from $h^+_l$ to $h^-_{l'}$ between loads $l$ and $l'$.
It is assumed that sufficient regular trucks are available to perform the traditional legs.
The corresponding costs will be captured in the objective function, but the regular truck routes are not modeled explicitly.
This is motivated by the fact that, in practice, the first- and last-mile problems are not very constrained.
Third, it is decided at which time each load is picked up.
It is assumed that every load $l \in L$ admits a flexibility of $\Delta \ge 0$ around the planned departure time $r(l)$, leading to a time window of $[r(l)-\Delta, r(l)+\Delta]$ for pickup.
This time window is translated to $h^+_l$, $h^-_l$, and $d(l)$ according to the travel times to maintain this flexibility throughout.
The travel times include time for loading and unloading the autonomous truck, which is assumed to be $S \ge 0$.
A solution is feasible if each load is served autonomously or directly, all implied autonomous legs are covered by autonomous truck routes, and the autonomous truck routes are feasible with respect to time.
Note that it is always feasible to replicate the current situation by serving all loads directly and not using any autonomous trucks.

\paragraph{Location Graph}
Before defining the objective, it is convenient to define a \emph{location graph}.
The location graph models all relevant locations and potential connections in the ATHN.
Let the location graph be denoted by the directed graph $G = (V, A)$.
Vertex set $V$ contains a vertex for every hub location, and two vertices for every load $l\in L$ that correspond to the origin $o(l)$ and the destination $d(l)$, respectively.
Arcs $a\in A$ are defined from every origin to the hubs (traditional first mile), between all the hubs (autonomous middle mile), from the hubs to every destination (traditional last mile), and between origin and destination directly (traditional direct delivery and empty return).
Note that the arcs between the hubs form a complete graph.
Every arc $a \in A$ is associated with a distance $c_a \ge 0$ and a travel time $\tau_a > 0$ obtained from \citet{OpenStreetMap2021}.
For convenience, the cost and travel time from $i \in V$ to itself are defined as 0.

\paragraph{Objective}
The objective is to serve all loads at minimum cost.
While any non-negative arc-additive cost structure is supported, this paper will define cost as the total distance in traditional mileage equivalent.
Autonomous trucks incur a cost of $(1-\alpha)c_a$ for every arc $a \in A$ on their routes, including arcs that represent empty relocations.
The parameter $\alpha \in [0, 1]$ discounts the autonomous distance to correct for reduced labor cost.
A direct trip for load $l \in L$ has a cost equal to its distance of $c_{o(l)d(l)} + c_{d(l)o(l)}$.
Note that the discount does not apply to regular trucks.
Finally, each first/last-mile arc $a \in A$ is assigned a cost of $\frac{1}{1-\beta} c_a$.
The parameter $\beta \in [0,1)$ represents the first/last-mile inefficiency, which assumes that a fraction $\beta$ of the first/last-mile route mileage would be empty.
The factor $\frac{1}{1-\beta}$ increases the cost of the first/last-mile arcs to compensate for the fact that these routes are not modeled explicitly.
The total objective is the sum of the above components and can be interpreted as the total distance measured in equivalent traditional mileage.

\section{Methodology}
\label{sec:methodology}

This section introduces the methodology that enables a large-scale data-driven study to quantify the impact of self-driving trucks.
The nomenclature for this section is summarized by Table~\ref{tab:symbols}.
Practical assumptions and preprocessing steps lead to a model that is easy to implement, can immediately be solved by blackbox solvers, and is highly extensible.
The section ends by providing a practical guide to enable regional and temporal analysis in this framework, which requires only minor modifications to the input and the model.

\begin{figure*}[!tp]
	\scriptsize
	\centering
	\begin{tabular}{p{0.05\textwidth}p{0.90\textwidth}}
		\toprule
		Symbol & Definition \\
		\midrule
		\multicolumn{2}{l}{\textbf{Sets and Graphs}} \\
		$T$ & set of tasks, each task $t \in T$ corresponds one-to-one to a load $l(t) \in L$, and represents serving this load on the ATHN with pickup time $p(t)$ at its origin hub.\\
		$\bar{G}$ & $=(\bar{V}, \bar{A})$, task graph that models the sequence of tasks.\\
		$\bar{V}$ & set of vertices $\{0, \hdots, n+1\}$ with source $0$, sink $n+1$, and tasks $1, \hdots, n$.\\
		$\bar{A}$ & set of task arcs, each arc $(t, t') \in \bar{A}$ indicates that vertex $t \in \bar{V}$ is followed immediately by vertex $t' \in \bar{V}$.\\
		\multicolumn{2}{l}{\textbf{Parameters}} \\
		$\bar{\tau}_a$ & duration of task arc $(t,t') \in \bar{A}$, $\tau_a > 0$, which consists of loading a truck for task $t$, driving between hubs, unloading, and relocating to the origin hub of task $t'$.\\
		$\bar{C}_t$ & baseline cost for serving load $l(t)$ directly with a regular truck.\\
		$\bar{c}_a$ & cost of task arc $a \in \bar{A}$, which is the difference between serving load $l(t)$ compared to the baseline.\\
		$M_{tt'}$ & $=p(t) - p(t') + 2\Delta + \bar{\tau}_{tt'}$, for $t,t' \in T$, sufficiently large big-M for Constraints~\eqref{eq:athn:timemtz}.\\
		\multicolumn{2}{l}{\textbf{Variables}} \\
		$x_t$ & continuous variable that indicates the start time of task $t \in T$.\\
		$y_a$ & binary variable that takes value one if $a \in \bar{A}$ is selected (i.e., the corresponding tasks are performed sequentially by the same vehicle), and zero otherwise.\\
		\bottomrule
	\end{tabular}
	\captionof{table}{Nomenclature Methodology.}%
	\label{tab:symbols}%
\end{figure*}

\paragraph{Task Graph}
The optimization model considers the problem of optimizing ATHN operations from the perspective of scheduling \emph{tasks} for autonomous trucks.
Similar transformations are common in the arc-routing literature (e.g., see \citep{BlackEtAl2013-TimeDependentPrize}).
One task $t \in T$ is created for every load $l \in L$.
If an autonomous truck performs a task, it means that the corresponding load is served through the autonomous network, and this truck serves the middle mile.
If a task is not performed by any autonomous truck, this means that the corresponding load is served directly by a regular truck.
Note that while performing tasks is optional, all loads are served in the end: performing a task only indicates that the task is served autonomously.
Appropriate benefits will be assigned to performing tasks to match the cost structure in Section~\ref{sec:problem}.

\begin{figure}[!t]
	\centering
	\includegraphics[width=0.8\textwidth]{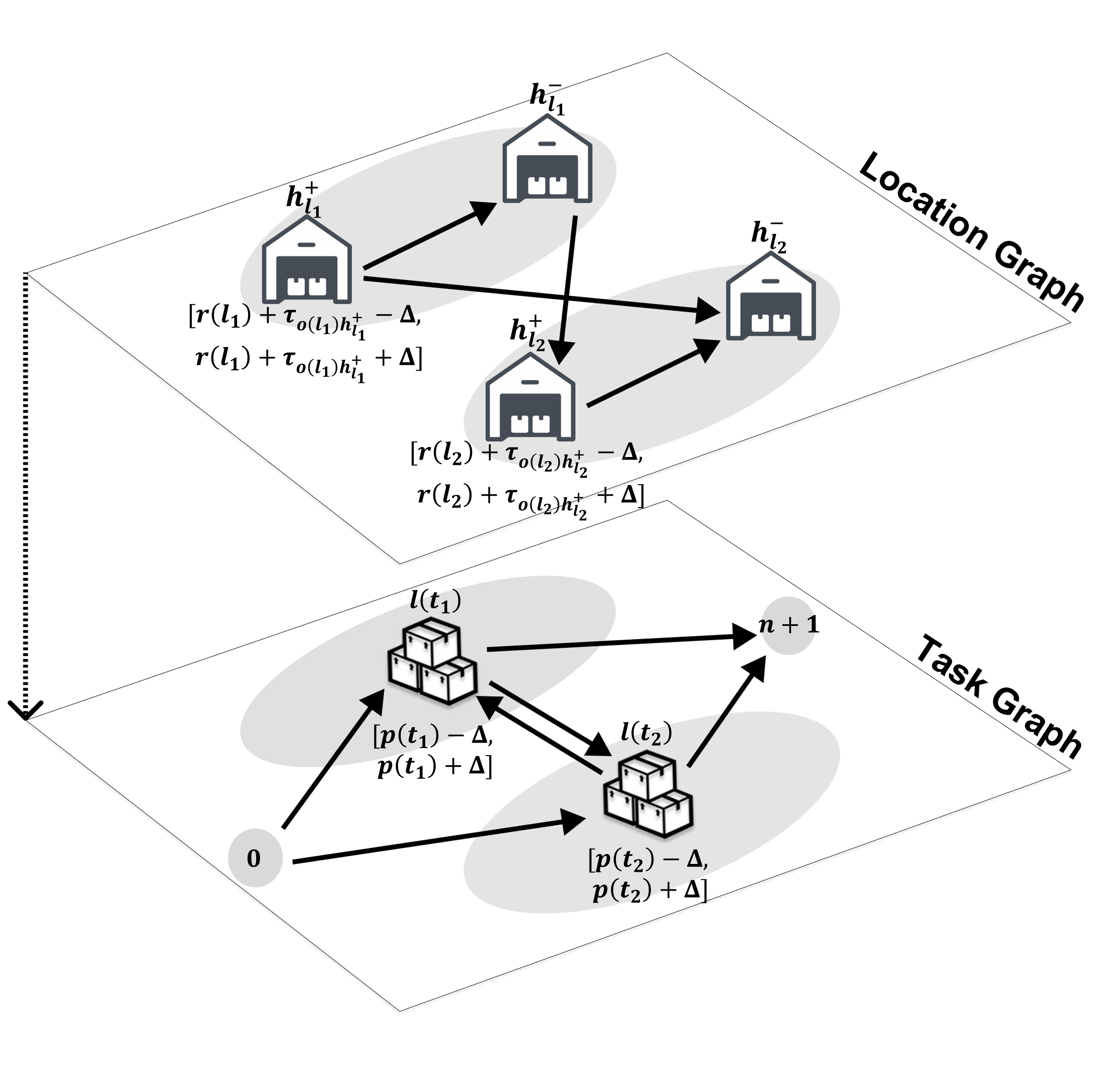}
	\caption{Constructing the Task Graph from the Location Graph.}
	\label{fig:taskgraph}
\end{figure}

To capture this perspective, the location graph is transformed into a directed \emph{task graph} $\bar{G} = (\bar{V}, \bar{A})$ in which the nodes are tasks and the arcs indicate the sequence of tasks performed by the same autonomous truck.
The set $\bar{V}$ includes a source node $0$ where each sequence starts, and a sink node $n+1$ where it ends.
As the nodes now represent tasks instead of locations, the number of nodes in the task graph is typically larger than the number of nodes in the location graph.
Arcs are defined from the source to the tasks, between the tasks (bi-directional), and from the tasks to the sink.
Figure~\ref{fig:taskgraph} provides an illustrative example.
The location graph shows four hubs and two loads $l_1$ and $l_2$ associated with tasks $t_1$ and $t_2$, respectively.
Visiting node $t_1$ means that an autonomous truck loads at origin hub $h^+_{l_1}$, drives to destination hub $h^-_{l_1}$, and unloads there.
It also implies that the first and last miles are performed by regular trucks (not pictured).
After that, the autonomous truck may either perform another task $t_2 \in T$, which first requires a relocation from $h^-_{l_1}$ to $h^+_{l_2}$, or it may end its sequence.
Loads for which the corresponding task is not covered are served by a direct trip with a regular truck (not pictured).
This means that all operations in the ATHN are captured in the task graph by a set of paths from the source to the sink, where each path corresponds to an autonomous truck.

\paragraph{Routes and Costs}
Optimizing ATHN operations now amounts to choosing a set of feasible autonomous truck routes that minimize the total cost.
A route is defined as a simple path in the task graph from source to sink, together with a starting time for every task.
Arcs between tasks $t_1, t_2 \in T$ model the passage of time between picking up loads $l_1$ and $l_2$, respectively.
That is, the duration is defined as $\bar{\tau}_{t_1t_2} = S + \tau_{h^+_{l_1} h^-_{l_1}} + S + \tau_{h^-_{l_1} h^+_{l_2}} > 0$, which sums the time for loading, performing the middle mile of load $l_1$, unloading, and relocating to the starting point of load $l_2$.
Task $t_1$ must start in the correct time window, which is obtained by shifting the original time window of load $l_1$ by the time it takes to perform the first mile.
This time window is given by $[p(t_1) - \Delta, p(t_1) + \Delta]$, where $p(t_1) = r(l_1) + t_{o(l_1) h^+_{l_1}}$.

Not covering task $t_1 \in T$ is associated with a constant baseline cost of $\bar{C}_{t_1}$ for performing a direct trip.
This value is given by $\bar{C}_{t_1} = c_{o(l_1) d(l_1)} + c_{d(l_1) o(l_1)}$.
If task $t_1$ \emph{is} performed, the cost on the outgoing arc replaces the baseline cost with the appropriate costs for serving the load autonomously.
More precisely, if task $t_1$ appears in a sequence followed by task $t_2 \in T$, the cost $\bar{c}_{t_1t_2}$ of arc $(t_1, t_2) \in A$ is defined as follows:
\begin{equation}
	\bar{c}_{t_1t_2} = \underbrace{\frac{1}{1-\beta} c_{o(l_1) h^+_{l_1}}}_{\textrm{first mile}} + \underbrace{(1-\alpha) c_{h^+_{l_1} h^-_{l_1}}}_{\textrm{middle mile}} + \underbrace{\frac{1}{1-\beta} c_{h^-_{l_1} d(l_1)}}_{\textrm{last mile}} + \underbrace{(1-\alpha) c_{h^-_{l_1} h^+_{l_2}}}_{\textrm{relocation}} - \underbrace{\bar{C}_{t_1}}_{\textrm{direct}}.
\end{equation}
Along the same lines, source arcs $a\in \bar{A}$ have cost $\bar{c}_a = 0$ and sink arcs omit the relocation term.
Note that $\bar{c}_a < 0$ when an autonomous delivery is preferred over a direct delivery, which encourages the task to be performed.

\paragraph{Optimization Model}

The optimization problem can now be stated as follows:
\begin{mini!}
	%
	{}
	%
	{\sum_{t \in T} \bar{C_t} + \sum_{a \in \bar{A}} \bar{c}_a y_a, \label{eq:athn:obj}}
	%
	{\label{formulation:athn}}
	%
	{}
	%
	%
	\addConstraint
	{\sum_{a \in \delta^+_t} y_a}
	{\le 1 \quad \label{eq:athn:visit}}
	{\forall t \in T,}
	\addConstraint
	{\sum_{a \in \delta^+_t} y_a}
	{= \sum_{a \in \delta^-_t} y_a \label{eq:athn:flowbalance}}
	{\forall t \in T,}
	\addConstraint
	{\sum_{a \in \delta^+_0} y_a}
	{\le K, \label{eq:athn:vehicles}}
	{}
	\addConstraint
	{x_{t'}}
	{\ge x_t + \bar{\tau}_{tt'} - M_{tt'} (1-y_a),\quad \label{eq:athn:timemtz}}
	{\forall t, t' \in T, (t,t') \in \bar{A}}
	\addConstraint
	{x_t}
	{\in [p(t)-\Delta, p(t)+\Delta] \label{eq:athn:x}}
	{\forall t \in T,}
	\addConstraint
	{y_a}
	{\in \mathbb{B} \label{eq:athn:y}}
	{\forall a \in \bar{A}.}
\end{mini!}%

\noindent Let $x_t \in [p(t)-\Delta, p(t)+\Delta]$ be the start time of task $t \in T$.
The variable $y_a \in \mathbb{B}$ is the flow on arc $a \in \bar{A}$, i.e., it takes value one if the tasks are performed sequentially by the same vehicle and zero otherwise.
For convenience, let $\delta^+_v$ and $\delta^-_v$ denote the out-arcs and in-arcs of $v \in \bar{V}$, respectively.
Problem~\eqref{formulation:athn} then models the optimization of ATHN operations.
Objective~\eqref{eq:athn:obj} minimizes the system cost as discussed above.
Constraints~\eqref{eq:athn:visit} require that each task is performed at most once, and Constraints~\eqref{eq:athn:flowbalance} ensure flow conservation.
The number of vehicles is limited by Constraint~\eqref{eq:athn:vehicles}.
Constraints~\eqref{eq:athn:timemtz} are \citet*{MillerEtAl1960-IntegerProgrammingFormulation} constraints that model the passage of time and eliminate cycles.
It is straightforward to show that the constants
\begin{equation}
	\label{eq:bigM}
	M_{tt'} = \bar{\tau}_{tt'} - \underbrace{\left(p(t') - \Delta\right)}_{\textrm{lowerbound on $x_{t'}$}} + \underbrace{\left(p(t) + \Delta\right)}_{\textrm{upperbound on $x_t$}}
\end{equation}
are sufficiently large to make the constraint inactive when $y_a=0$.
Finally, Equations~\eqref{eq:athn:x}-\eqref{eq:athn:y} define the variables.
For a consistent analysis, the solution is postprocessed to shift the start times to as early in time as possible.

\subsection{Acceleration Techniques}
The size of Problem~\eqref{formulation:athn} can be reduced significantly by recognizing that many arcs $(t,t') \in \bar{A}$ are either trivially time-feasible because task $t'$ is planned much later than task $t$, or trivially time-infeasible because task $t'$ is planned much earlier than task $t$.
These observations are formalized in the following proposition.

\begin{myprop}[Preprocessing Rules]
	\label{prop:preprocessing}
	Let $a_t = p(t)-\Delta$ and $b_t = p(t)+\Delta$ be the earliest and latest possible start time of task $t \in T$, respectively.
	The following preprocessing rules are valid for arc $(t,t') \in \bar{A}$ between two tasks $t,t' \in T$.
	\begin{enumerate}
		\item \textbf{Arc is always time feasible:} $b_t + \bar{\tau}_{tt'} \le a_{t'}$ $\Rightarrow$ remove Constraint~\eqref{eq:athn:timemtz} for arc $(t,t')$.
		\item \textbf{Arc is never time feasible:} $a_t + \bar{\tau}_{tt'} > b_{t'}$ $\Rightarrow$ remove arc $(t,t')$ from $\bar{A}$.
	\end{enumerate}
\end{myprop}
\begin{proof}
	By definition, $x_t$ is only allowed to take values in $x_t \in [a_t, b_t]$.
	The condition in the first rule implies $x_t + \bar{\tau}_{tt'} \le b_t + \bar{\tau}_{tt'} \le a_{t'} \le x_{t'}$ $\Leftrightarrow$ $x_{t'} \ge x_t + \bar{\tau}_{tt'}$ for all feasible values of $x_t$ and $x_{t'}$.
	It follows that the time constraint is redundant and can be removed.
	Similarly, the condition in the second rule implies $x_t + \bar{\tau}_{tt'} \ge a_t + \bar{\tau}_{tt'} > b_{t'} \ge x_t'$ $\Leftrightarrow$ $x_t' < x_t + \bar{\tau}_{tt'}$.
	It follows that $y_{tt'} = 1$ would violate Constraint~\eqref{eq:athn:timemtz}.
	As such, $y_{tt'}$ may be set to zero, which is achieved by simply removing the arc.
	Hence, these preprocessing rules are valid.
\end{proof}

Both preprocessing rules eliminate time constraints, which can make them incredibly powerful.
If all time Constraints~\eqref{eq:athn:timemtz} are eliminated, then the $x$-variables~\eqref{eq:athn:x} are automatically satisfied, and the remainder of Problem~\eqref{formulation:athn} can be seen as a minimum-cost network flow problem.
The only constraints that are not in standard form are Constraints~\eqref{eq:athn:visit} and \eqref{eq:athn:vehicles}, but they take the form of node capacities that can be handled through node splitting \citep{AhujaEtAl1993-NetworkFlowsTheory}.
It is well-known that the min-cost flow problem exhibits the integrality property, and can be solved in polynomial time by linear programming.
As the flexibility $\Delta$ decreases, the preprocessing rules become more effective, and Problem~\eqref{formulation:athn} gets closer to a minimum-cost network flow problem.
In fact, this situation is reached for the no-flexibility $\Delta = 0$ case, when every arc $(t,t')\in \bar{A}$ either satisfies Rule 1 or Rule 2 and all time constraints are eliminated.
Informally, it is easier to optimize ATHN operations when there is less flexibility, to the point where it becomes provably easy without flexibility.

\paragraph{MIP Start}
In addition to preprocessing, this paper will try to improve the optimization process by providing the solver with an initial feasible solution.
This solution is known as a \emph{MIP start} and provides an upper bound that can assist the branch-and-bound process.
Regardless of the flexibility $\Delta$, a feasible solution can be calculated efficiently by solving the case when flexibility is set to zero.
The calculated solution is then used as a starting point for the actual problem.
To avoid the overhead of building an additional model, the solver is provided a partial MIP start of only $x_t = p(t)$ for all $t\in T$, which is sufficient to find the same solution.

The fact that $\Delta=0$ is easy to solve and guarantees a valid upper bound is specifically because the trucks are autonomous.
The main technical difference is that autonomous trucks are completely interchangable, while human-driven trucks need to be distinguished to ensure that drivers return to their specific starting point \citep{ArunapuramEtAl2003-VehicleRoutingScheduling} or that they do not exceed the maximum driving time \citep{GronaltEtAl2003-NewSavingsBased}.
Network flow relaxations that aggregated drivers have been used to derive lower bounds \citep{GronaltEtAl2003-NewSavingsBased}, but it is not obvious how to transform the outcome into a feasible solution.
For autonomous trucks these human factors do not apply, which enables the framework in this paper.

\subsection{Regional and Temporal Decomposition}
\label{sec:reg_temp}
The framework in this paper is easily extended to perform regional and temporal decomposition.
This requires only minor modifications to the input and to the model.

\paragraph{Regional Decomposition}
The goal of the regional decomposition is to compare the global optimization of ATHN operations to a situation in which each region (e.g., the South of the US) has dedicated autonomous trucks that only pick up loads that start in that region.
In this case, trucks can serve loads within their region and loads that are moving out of the region.
However, after leaving the region to drop off a load, the truck has to return before it can perform another task.
The model will be modified to jointly optimize how trucks are assigned to regions and how to operate the ATHN under these restrictions.
Performing an analysis in this setting helps answer questions about the scale at which autonomous trucks are effective and where they should be deployed.

Regional decomposition can easily be performed by filtering arcs from the task graph.
First assign every task $t\in T$ to a region based on the location of the origin hub $h^+_{l(t)}$.
Next, remove all arcs $(t,t')\in \bar{A}$ between tasks $t, t' \in T$ if the regions are different.
It follows that when a flow reaches a task that starts from a specific region, there is no way to reach tasks that start from a different region, as intended.
The amount of flow from the source to each region represents the amount of trucks that are assigned to that region.
The truck assignments and operations are then jointly optimized by solving this filtered instance of Problem~\eqref{formulation:athn}.

\paragraph{Temporal Decomposition}
The goal of the temporal decomposition is to plan ATHN operations on a rolling horizon (e.g., one week at a time), rather than for a full period at once (e.g., four weeks).
Optimizing over a shorter horizon requires less information and is easier computationally.
However, the model does not explicitly rebalance trucks at the end of the period.
This means that optimizing myopically may leave the trucks ill-positioned for the next period.
The temporal decomposition can be used to explore these trade-offs.

\begin{figure}[!tp]
	\centering
	\includegraphics[width=0.8\linewidth]{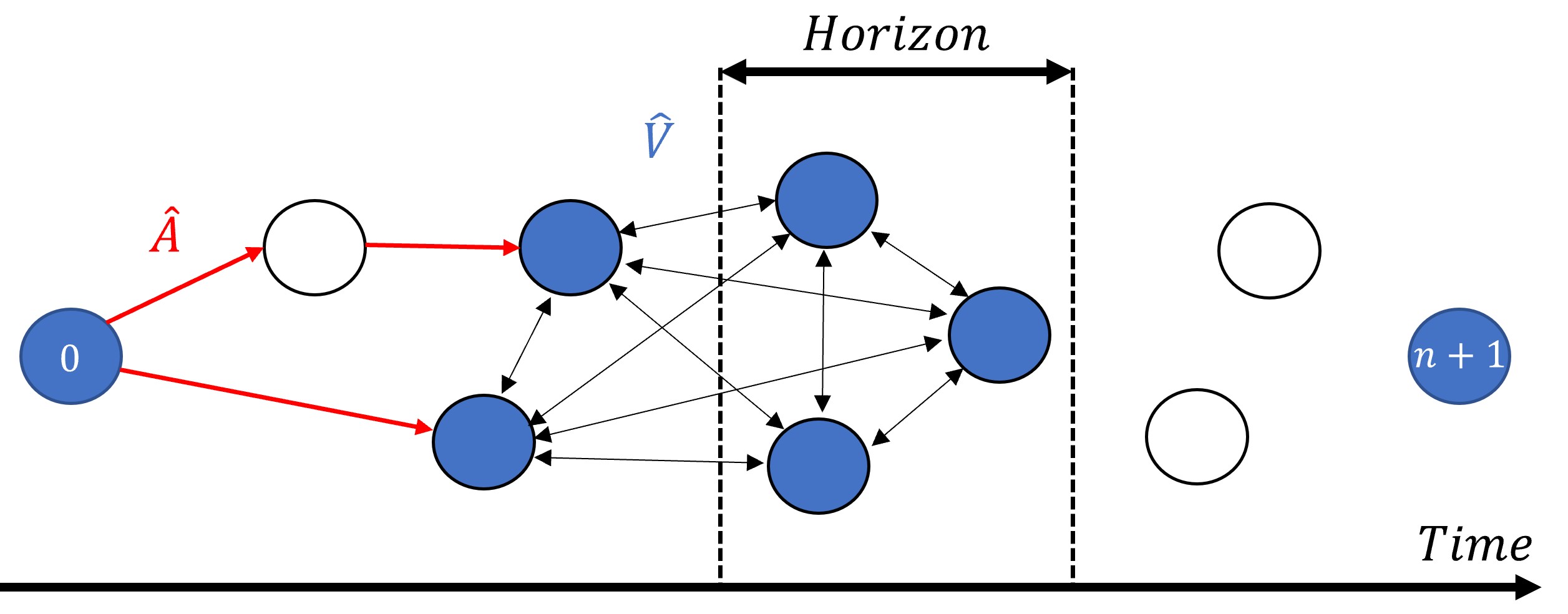}
	\caption{Arc Filtering for Temporal Decomposition ($\hat{A}$ in red and $\hat{V}$ in blue)}
	\label{fig:temporal_decomp}
\end{figure}

Implementing a rolling horizon for Problem~\eqref{formulation:athn} is relatively straightforward.
A practical way to do so is by reusing the existing structures and modifying the model as little as possible.
Figure~\ref{fig:temporal_decomp} provides an illustrative example.
First, build the task graph for the full period.
For a given horizon, identify the arcs $\hat{A}$ of the partial routes created previously (without sink arcs).
Fix these arcs $a \in \hat{A}$ to $y_a = 1$ in the optimization model to stay consistent with the past.
To plan for the current horizon, only the following nodes $\hat{V}$ are relevant: the source, the sink, the current route endpoints, and the tasks that start during the horizon.
Now filter the task graph to only keep the arcs in $\hat{A}$ and the arcs in the subgraph induced by $\hat{V}$.
This makes it impossible to plan outside of the horizon.
The model is solved and the steps are repeated until the full period is planned.

\section{Case Study}
\label{sec:casestudy}

To quantify the impact of autonomous trucking on a realistic transportation network, a case study is presented for the dedicated transportation business of Ryder System, Inc. (Ryder).
Ryder is one of the largest transportation and logistics companies in North America, and provides fleet management, supply chain, and dedicated transportation services to over 50,000 customers.

\paragraph{Data}
Ryder has provided a dataset that is representative for its dedicated transportation business in the US, reducing the scope to orders that are strong candidates for automation.
The case study focuses on the \emph{challenging orders} that currently consist of a single delivery followed by an empty return trip.
These orders are highly inefficient and contribute significantly to the overall empty mileage, such that ATHN can potentially have a big impact.
The challenging orders also allow for a clean comparison to the current situation: These are orders for which Ryder was unable to find a backhaul, and returning empty after delivery is how they would be served in practice.
The challenging orders are converted into loads with an origin, destination, and planned departure time.
The ATHN operations are optimized for loads that start during the first four weeks of October 2019.
This corresponds to 6842 loads, with an average distance of 390 km (242 mi).

\paragraph{Network Design}
To design an effective network, it is important to select hubs that are 1. close to frequently used origins and destinations to minimize the first/last mile, and 2. easily accessible from the highway to enable automation between the hubs.
This is achieved by first using the K-means algorithm in Scikit-learn \citep{PedregosaEtAl2011-ScikitLearnMachine} to cluster the origins and destinations into the desired number of hubs, based on data from October to December 2019.
Next, the centroids are mapped onto the closest US truck stop obtained from the \citet{DepartmentTransportation2019-TruckStopParking}, according to haversine distance.
The origin and destination hubs $h^+_l \neq h^-_l$ for load $l \in L$ are chosen to minimize the value of $c_{o(l)h^+_l} + (1-\gamma) c_{h^+_l h^-_l} + c_{h^-_ld(l)}$, where $\gamma \in [0,1]$ is a discount factor for autonomous trucks.
For $\gamma=0$ this minimizes the total distance, for $\gamma=1$ this minimizes the first/last-mile distance, and $\gamma\in (0, 1)$ minimizes a combination of the two.
By taking both the origin and the destination into account, the rule allows for assigning hubs in the right direction that are not necessarily the closest.
Figure~\ref{fig:athn_design_100} visualizes the 100-hub design for $\gamma=40\%$, in which the area of each hub is proportional to the number of loads that are assigned to it.
It can be seen that many loads are concentrated in the South (purple) and in the Northeast (red).
\begin{figure}[!t]
	\centering
	\includegraphics[width=\linewidth]{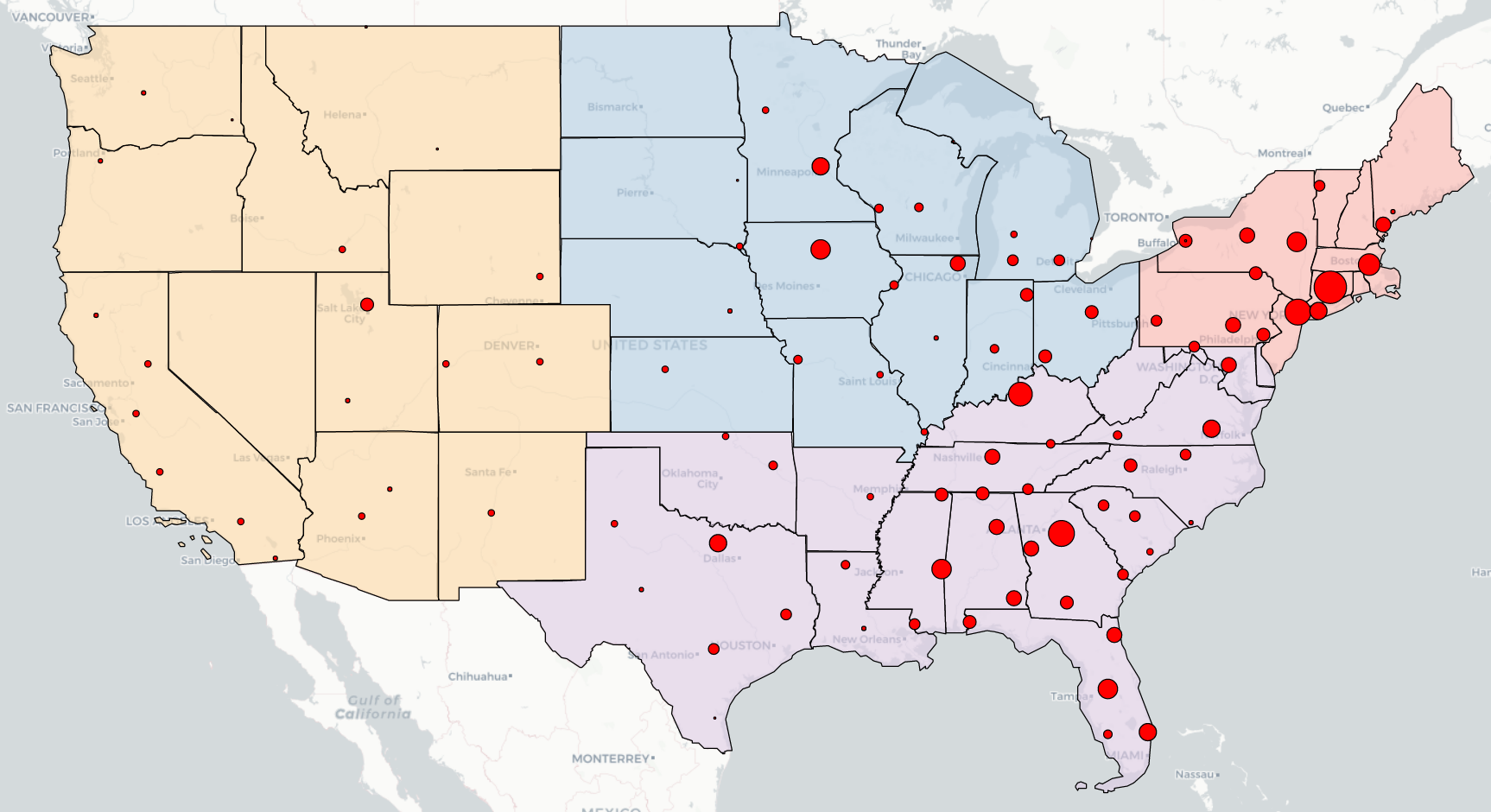}
	\caption{ATHN Design for 100 Hubs (Circle Area Proportional to Number of Assigned Loads)}
	\label{fig:athn_design_100}
\end{figure}

\paragraph{Experimental Settings}
Table~\ref{tab:base_parameters} provides an overview of the baseline parameter values used in the case study.
Various sensitivity analyses will be performed to observe how these parameters affect the system.
The baseline includes two autonomous discount factors: $\alpha=25\%$ and $\alpha=40\%$.
This results in a conservative estimate of the benefits of autonomous trucking, which is predicted to be 29\% to 45\% cheaper per mile \citep{EngholmEtAl2020-CostAnalysisDriverless}.
All experiments use a consistent hub-assignment rule with autonomous discount factor $\gamma=40\%$ to ensure that the results are comparable.
A higher value of $\gamma$ tends towards load paths that include more autonomous mileage.
The baseline also uses multiple values of $K$ to observe the impact of increasing availability of autonomous trucks, with the value $K=100$ as the standard.
These instances are solved sequentially by increasing the right-hand side of Constraint~\eqref{eq:athn:vehicles} and reoptimizing.
All steps in Section~\ref{sec:methodology} are implemented in Python~3.9 and the ATHN operations are optimized with Gurobi~9.5.2.
Gurobi is given three hours of solving time for each instance, unless stated otherwise.
Each experiment is run on a Linux machine with dual Intel Xeon Gold 6226 CPUs on the PACE Phoenix cluster \citep{PACE2017-PartnershipAdvancedComputing}, using a single node with 24 cores and 192GB of RAM.
If memory is insufficient, the experiment is repeated on a machine with 384GB of RAM.
Note that if high-memory machines are not available, memory could also be traded for computing time by reducing the number of parallel threads.

\begin{figure*}[!tp]
	\scriptsize
	\centering
	\begin{tabular}{ll}
		\toprule
		Parameter & Value \\
		\midrule
		$n$ & 6842 loads\\
		$\lvert V_H \rvert$ & 100 transfer hubs\\
		$K$ & $\in \{0, 50, \underline{100}, 150, 200, 250\}$ autonomous trucks\\
		$\Delta$ & 1 hour pickup-time flexibility\\
		$S$ & 30 minutes autonomous truck loading/unloading time\\
		$\alpha$ & $\in \{25\%, 40\%\}$ discount for autonomous mileage\\
		$\beta$ & 25\% first/last-mile inefficiency\\
		$\gamma$ & 40\% discount for autonomous mileage during hub-assignment\\
		Preprocessing & Applied\\
		MIP start & Disabled\\
		Solver time limit & 3 hours\\
		\bottomrule
	\end{tabular}
	\captionof{table}{Baseline Parameter Values for the Case Study.}%
	\label{tab:base_parameters}%
\end{figure*}

\section{Baseline Results}
\label{sec:baseline}

This section discusses the baseline results for the case study.
It first presents computational results to demonstrate that the presented framework can handle large-scale systems.
Next, it analyzes the impact of autonomous trucking for the case study.

\subsection{Computational Results}
\label{sec:computational}

Table~\ref{tab:base_computation} presents the computational results for the baseline instances.
The instances differ by the discount for autonomous mileage $\alpha$ and the number of vehicles $K$.
As described in the previous section, the instances for different $K$ are run sequentially and reuse the same model, as would be done in practice to study the system.
The `LP Relaxation' columns present the time to solve the linear programming relaxation (before cuts) and the corresponding root gap.
The `Branch and Bound' columns summarize the full branch-and-bound process, reporting the number of nodes in the tree, the solution time (10,800 if the time limit of three hours is reached), and the final gap.
The table omits the time for building the model, which was less than six minutes, and the time for presolve, which took less than four minutes in all cases.

Despite the fact that each model has over 22M binary variables and close to 300k constraints, Gurobi is able to find optimal solutions in most cases.
For $K=0$, the problem is trivial and is solved immediately in presolve.
The cases with fewer vehicles are challenging to the solver, presumably because the tasks are packed more densely into the schedule, as will be discussed in Section~\ref{sec:sens:number_trucks}.
Only the $K=50$ cases where not solved to optimality, and remain at 0.03\% and 0.18\% gap.
Both for $\alpha=25\%$ and $\alpha=40\%$ the solver tends to keep adding cutting planes rather than branch.
This strategy is successful to solve all other instances to optimality within the time limit.
The only instance that stands out is $\alpha=40\%$ and $K=100$, which explores 3382 nodes in the branch-and-bound tree.
For this instance, the log shows that a gap of 0.01\% is found after 4845 seconds.
When the gap is still at 0.01\% at 5921 seconds, the solver decides to start branching to close the gap.
This behavior can likely be explained by symmetry in the solution space, e.g., if two vehicles swap half of their tasks, the solution is likely to be of similar quality.
The result is that good solutions are found quickly, but it takes a substantial number of cuts or branches to find the optimum.
Overall, the computations for the baseline instances show that the proposed methodology can find optimal or close to optimal solutions in a short amount of time compared to the planning horizon. 

\begin{table}[!t]
	\centering
	\begin{tabular}{rrrrrrr}
		\toprule
		\multicolumn{2}{c}{Parameters} & \multicolumn{2}{c}{LP Relaxation} & \multicolumn{3}{c}{Branch and Bound} \\
		\midrule
		$\alpha$ & $K$     & Seconds & Gap \% & Nodes & Seconds & Gap (\%) \\
		\midrule
		\multirow{6}[1]{*}{25\%} & 0     & 0  & 0  & 0     & 0  & 0  \\
		& 50    & 218  & 19.50  & 1     & 10,800 & 0.03  \\
		& 100   & 170  & 6.44  & 1     & 4,133  & 0  \\
		& 150   & 183  & 1.81  & 1     & 3,261  & 0  \\
		& 200   & 181  & 0.70  & 1     & 2,754  & 0  \\
		& 250   & 177  & 0.40  & 1     & 2,850  & 0  \\
		\midrule
		\multirow{6}[1]{*}{40\%} & 0     & 0  & 0  & 0     & 0  & 0  \\
		& 50    & 268  & 25.10  & 1     & 10,800 & 0.18  \\
		& 100   & 206  & 10.40  & 3,382  & 6,430  & 0  \\
		& 150   & 232  & 3.20  & 1     & 2,972  & 0  \\
		& 200   & 207  & 1.00  & 1     & 3,149  & 0  \\
		& 250   & 189  & 0.74  & 1     & 2,026  & 0  \\
		\bottomrule
	\end{tabular}
	\caption{Baseline Computation Statistics.}
	\label{tab:base_computation}%
\end{table}%

\paragraph{Using MIP starts}
Enabling MIP start forces the solver to construct an initial feasible solution before starting the search (Section~\ref{sec:methodology}).
Figure~\ref{fig:mipstart_gap} provides an example of the effect of enabling MIP start for the $\alpha = 25\%$ baseline with $K = 100$ trucks.
Note that these tests are run independently without reusing the model for $K=50$ as in the experiments above.
Without any guidance, Gurobi takes 520 seconds to report the first optimality gap of 33\%.
At 741 seconds, a significantly better solution of 1.25\% gap is found, and the problem is solved to optimality in under an hour.
When MIP start is enabled, it takes more time for the search proper to start, and the first gap is reported at 1139 seconds.
However, spending time to construct an initial feasible solution immediately leads to a gap of only 1.19\% because of the improved upper bound.
The full problem is solved in under one hour and 15 minutes.

\begin{figure}[!t]
	\centering
	\includegraphics[width=0.5\textwidth]{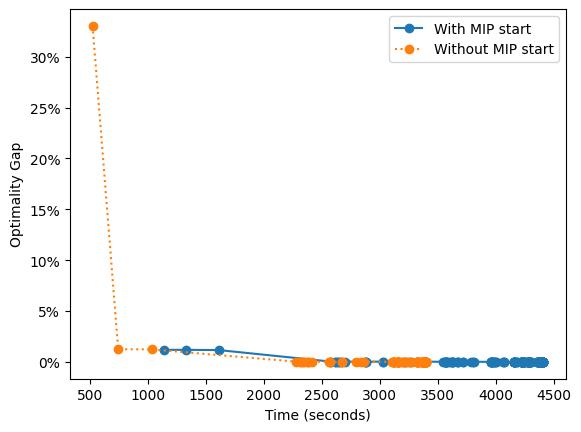}
	\caption{Optimality Gap over Time for $\alpha=25\%$ and 100 Trucks.}
	\label{fig:mipstart_gap}
\end{figure}

Two observations are made for the case study.
First, using a MIP start does not seem to improve solution time, but it does create a more predictable result.
Especially if the instance cannot be solved to optimality, the MIP start is more likely to produce a reasonable solution before the time limit.
Second, the small initial gap for the MIP start suggests that the initial solution is already of high quality.
Recall that this zero-flexibility $\Delta=0$ case can be seen as a min-cost flow problem, which gives practitioners the possibility to avoid commercial software and instead plan ATHN operations with highly-efficient open source solvers such as the LEMON graph library \citep{DezsoEtAl2011-LemonOpenSource}.
As most instances can be solved to optimality, MIP starts will only be enabled for the difficult large-flexibility instances in Section~\ref{sec:timeflex}.

\subsection{Impact of Autonomous trucking}
Figure~\ref{fig:base} presents the impact of autonomous trucking for the baseline, where autonomous mileage is discounted by either $\alpha=25\%$ or $\alpha=40\%$.
It is clear that introducing autonomous trucks leads to substantial benefits.
Figure~\ref{fig:base_cost} shows that the first 50 trucks already lower the operational cost of the system (including first/last miles) by the equivalent of more than one million traditional kilometers.
E.g., at \$1.25/km ($\approx$ \$2/mile) for traditional trucks, this corresponds to a value of about \$1.3M per four weeks or \$16.9M per year.
The percentage savings for the overall system are provided in Figure~\ref{fig:base_save_overall}.
These savings range from 20\% for 50 trucks in the more expensive scenario to 37\% for 250 trucks when autonomous trucking is less expensive.
It is interesting to observe that adding vehicles clearly satisfies the law of diminishing returns.
As more vehicles are added, more loads are served autonomously (Figure~\ref{fig:base_load}) and more savings are obtained (Figure~\ref{fig:base_save_overall}), but the benefits level out at about 100 trucks for the Ryder case study.
Note that this is a relatively small number of trucks compared to the 6842 loads, which reflects the fact that autonomous trucks can operate around the clock.
Figure~\ref{fig:base_save} looks at the savings percentage only for loads that are served autonomously.
It can be seen that, on average, loads that are served autonomously save between 31\% and 42\% in costs compared to traditional transportation.

\begin{figure}[!t]
	\centering
	\begin{subfigure}{0.47\textwidth}
		\centering
		\includegraphics[width=\textwidth]{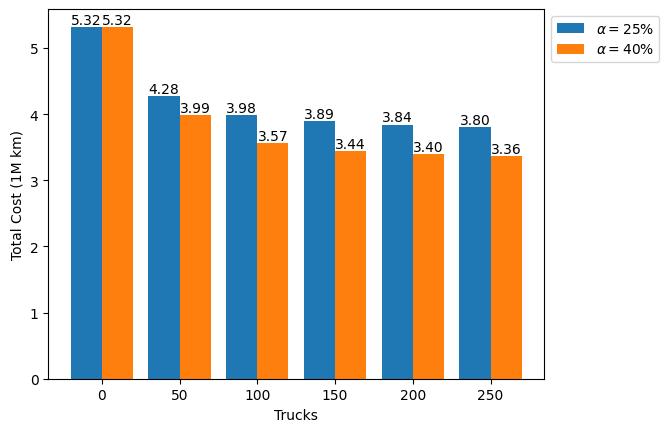}
		\caption{Total Cost}
		\label{fig:base_cost}
	\end{subfigure}
	\begin{subfigure}{0.48\textwidth}
		\centering
		\includegraphics[width=\textwidth]{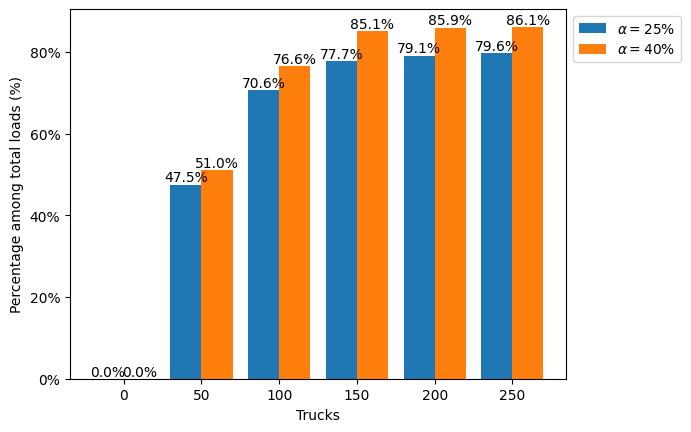}
		\caption{Loads Served Autonomously}
		\label{fig:base_load}
	\end{subfigure}
	\begin{subfigure}{0.47\textwidth}
		\centering
		\includegraphics[width=\textwidth]{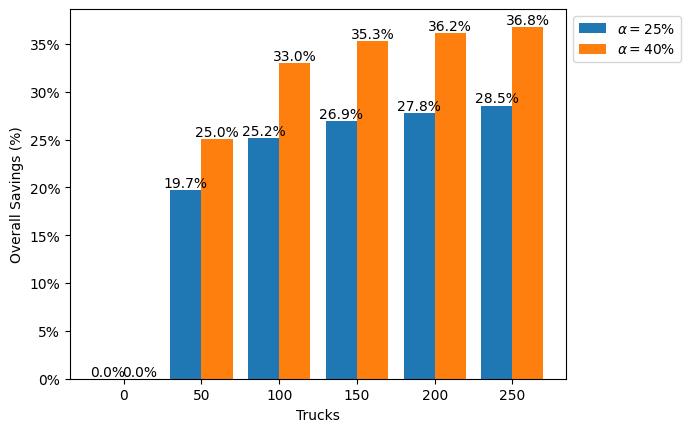}
		\caption{Total Savings}
		\label{fig:base_save_overall}
	\end{subfigure}
	\begin{subfigure}{0.47\textwidth}
		\centering
		\includegraphics[width=\textwidth]{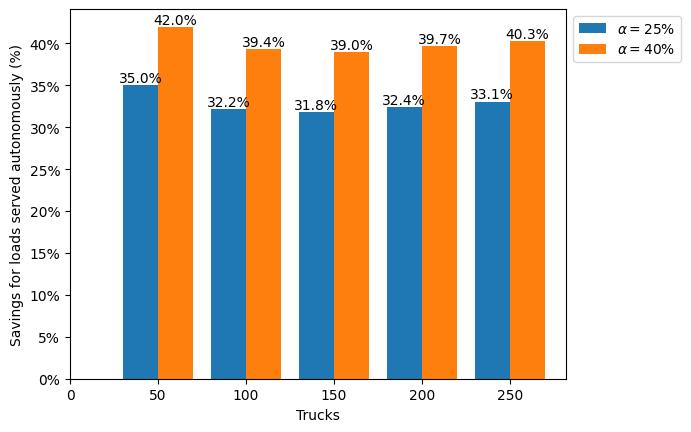}
		\caption{Savings Loads Served Autonomously}
		\label{fig:base_save}
	\end{subfigure}
	\caption{Baseline Results}
	\label{fig:base}
\end{figure}

Table~\ref{tab:base_stats} and Figure~\ref{fig:base_gantt} dive deeper into the results for $\alpha=25\%$ and 100 trucks specifically.
The table shows that the total distance driven in the ATHN (including empty miles) is 13.0\% lower than for the current system.
These savings are due to the flexibility of autonomous trucks that can operate throughout the night and never need to return home.
This allows for only 29\% empty miles on the autonomous middle mile, which is a substantial improvement over the rate of 50\% in the current system.
When labor cost reduction is taken into account, the savings increase to 25.2\%.
The truck schedule in Figure~\ref{fig:base_gantt} also shows that relocations are small compared to the work performed: Every row represents a single truck, where blue bars correspond to performing tasks, and the red bars correspond to driving empty.
The schedule is relatively tight, except for the four `gaps' during the weekends, in which not many loads are planned.
This is an artifact of the data that results from planning around people, Section~\ref{sec:results} will explore the value of increasing flexibility and allowing autonomous trucks to pick up loads on any day.

\begin{table}[!t]
	\centering
	\begin{tabular}{llclrccr}
		\toprule
		& Service & Loads & Segment & \multicolumn{1}{c}{Total km} & Empty & Cost Factor & \multicolumn{1}{c}{Cost} \\
		\midrule\midrule
		Current & Direct & 6842 & Full & 5,323,357 & 50\% & 1 & 5,323,357 \\
		\midrule\midrule
		ATHN & Autonomous & 4828 & Middle & 2,598,418 & 29\% & 0.75 & 1,948,814\\
		& & & First/last & 875,215 & 25\% & 1 & 875,215\\
		\cmidrule{2-8}
		& Direct & 2014 & Full & 1,159,133 & 50\% & 1 & 1,159,133\\
		\cmidrule{2-8}
		& Total & 6842 & & 4,632,766 & & & 3,983,162 \\
		\midrule\midrule
		\multicolumn{4}{r}{Savings ATHN compared to Current} & 690,591 & & & 1,340,195 \\
		& & & & 13.0\% & & & 25.2\%\\
		\bottomrule
	\end{tabular}
	\caption{Statistics for $\alpha=25\%$ and 100 Trucks.}%
	\label{tab:base_stats}%
\end{table}%

\begin{figure}[!t]
	\centering
	\includegraphics[width=\linewidth]{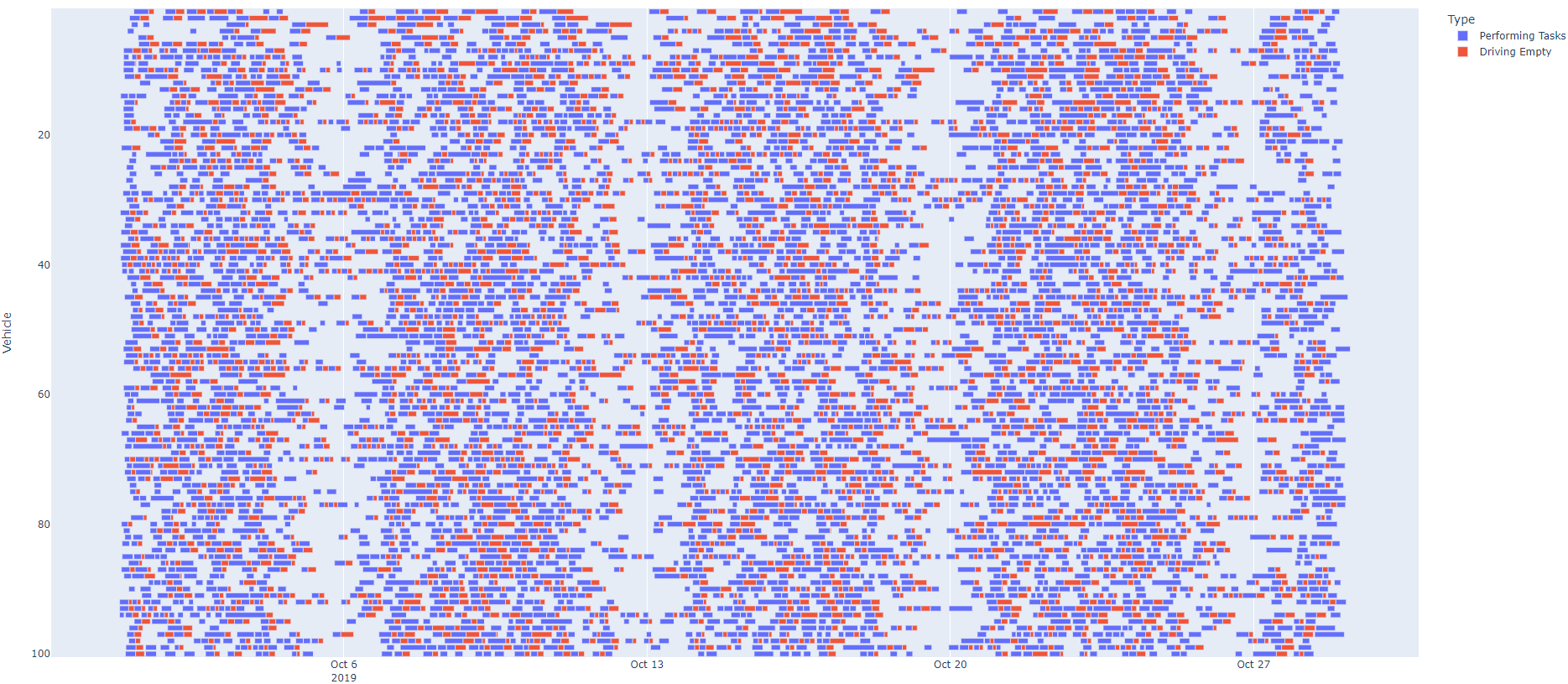}
	\caption{Truck Schedule for $\alpha=25\%$ and 100 Trucks.}
	\label{fig:base_gantt}
\end{figure}


\section{Sensitivity Analysis}
\label{sec:results}

The baseline results demonstrate significant benefits of autonomous trucking for the case study.
This section provides an extensive sensitivity analysis of these results.
It explores flexibility in pickup times, the trade-off between operating cost and autonomous truck utilization, the effect of regional and temporal decomposition, and various changes in input parameters.

\subsection{Pickup-time Flexibility}
\label{sec:timeflex}

Even for a limited pickup-time flexibility of one hour, the baseline showed significant benefits of ATHN compared to traditional transportation.
This section explores whether increasing the flexibility to up to 24 hours further improves network performance.
If there are major gains in efficiency, it may be worth negotiating new pickup times or new service-level agreements with the customers.
From a computational standpoint, Section~\ref{sec:computational} explained that large-flexibility instances are more difficult to solve.
For this reason, this section will rely heavily on MIP starts.

Figures~\ref{fig:delta_25} and \ref{fig:delta_40} present upper bounds (feasible solutions) and lower bounds for pickup-time flexibilities ranging from 0 to 24 hours.
In terms of computational performance, the figures show that the 0h and 1h instances can be solved to optimality, the 2h instances have a small remaining gap, and the 12h and 24h instances prove difficult to solve.
The exact optimality gaps are reported in Table~\ref{tab:delta_gap}.
It can be seen that for the difficult instances, using a MIP start really helps to obtain better solutions, e.g., reducing the gap from 30\% to 5\% in the $\alpha=40\%$ and $\Delta=24$h case.
Increasing the flexibility from 1h to 2h reduces the overall cost of the system by 0.9\% for $\alpha=25\%$ and 1.8\% for $\alpha=40\%$.
With the current algorithm, increasing the flexibility to 12h and 24h does not lead to better solutions (upper bounds).
However, the lower bounds show potential savings of up to 5.3\% compared to the $\Delta=1$h baseline.
This makes it an interesting direction for future research to develop methods that are effective for larger flexibility.
For $\alpha=40\%$ and $\Delta=24$h, the solver is able to improve over the MIP start solution.
The resulting schedule is presented by Figure~\ref{fig:delta_schedule}.
It shows that the increased flexibility creates a schedule that is more sparse and departs from the current weekly pattern shown by Figure~\ref{fig:base_gantt}.

\begin{figure}[!t]
	\centering
	\begin{subfigure}{0.47\textwidth}
		\centering
		\includegraphics[width=\textwidth]{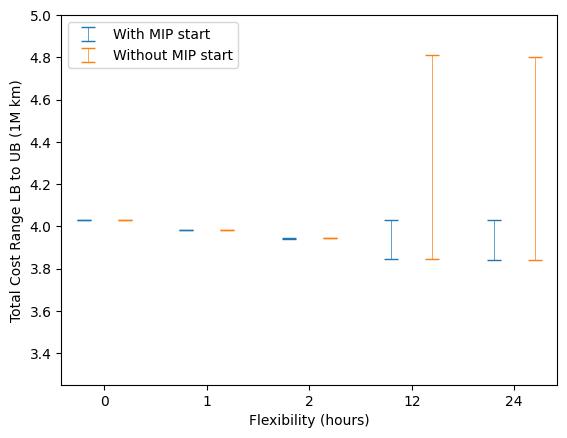}
		\caption{Case $\alpha=25\%$ and 100 trucks.}
		\label{fig:delta_25}
	\end{subfigure}
	\begin{subfigure}{0.47\textwidth}
		\centering
		\includegraphics[width=\textwidth]{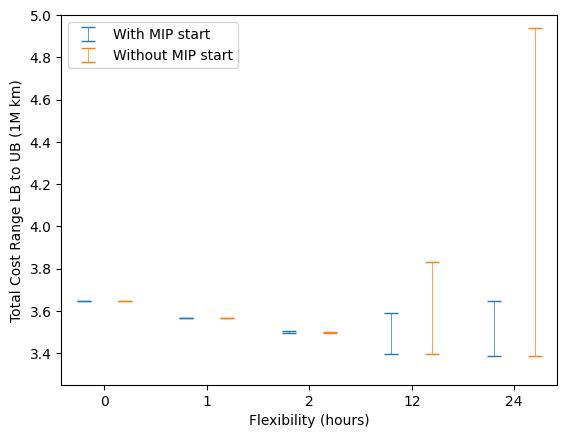}
		\caption{Case $\alpha=40\%$ and 100 trucks.}
		\label{fig:delta_40}
	\end{subfigure}
	\caption{Upper and Lower Bounds for Different Time Flexibilities $\Delta$.}
	\label{fig:delta}
\end{figure}

\begin{table}[!t]
	\centering
	\begin{tabular}{rrrrr}
		\toprule
		~ & \multicolumn{2}{c}{$\alpha=25\%$} & \multicolumn{2}{c}{$\alpha=40\%$} \\
		\midrule
		$\Delta$ & \multicolumn{1}{r}{Without MIP start} & \multicolumn{1}{r}{With MIP start} & \multicolumn{1}{r}{Without MIP start} & \multicolumn{1}{r}{With MIP start} \\
		\midrule
		2h  & 0.04\% & 0.07\% & 0.17\% & 0.20\%\\
		12h & 20.05\% & 4.52\% & 11.38\% & 5.50\% \\
		24h & 14.80\% & 4.69\% & 29.77\% & 4.62\% \\
		\bottomrule
	\end{tabular}%
	\caption{Optimality Gaps at Time Limit for Large-Flexibility Instances.}%
	\label{tab:delta_gap}
\end{table}

\begin{figure}[!t]
	\centering
	\includegraphics[width=\linewidth]{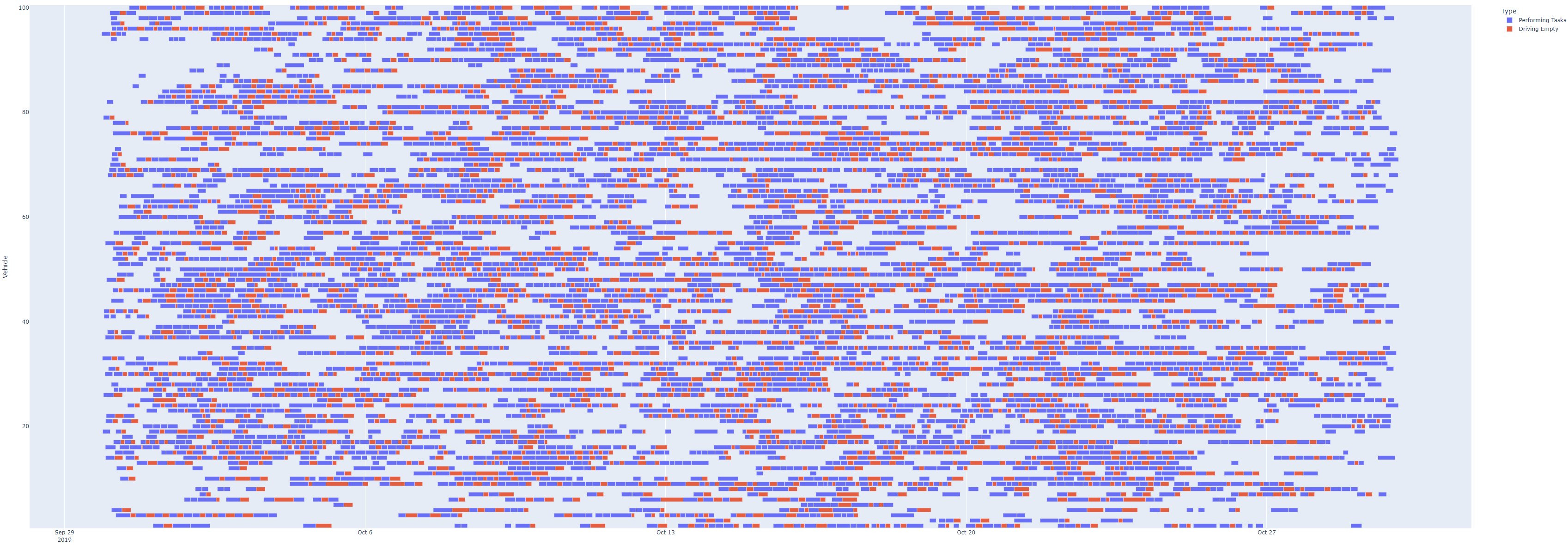}
	\caption{Truck Schedule for Increased Flexibility of $\Delta=24$h ($\alpha=40\%$, MIP start enabled)}
	\label{fig:delta_schedule}
\end{figure}


\subsection{Number of Autonomous Trucks}
\label{sec:sens:number_trucks}
This section explores the trade-off between operating cost and autonomous truck utilization in the ATHN.
Figure~\ref{fig:maxk_inactive} presents the total cost of the ATHN operations from zero to 250 trucks in increments of 10, and all solutions are within 0.25\% of optimality.
The \emph{inactivity} is defined as the percentage of time that autonomous trucks are waiting and not performing any tasks.
The convex cost curves clearly show that the first autonomous trucks will be the most impactful and will almost never be inactive.
This property makes it easier to run successful pilots and encourage the adoption of ATHN.
Cost improvements start leveling out as more vehicles are added to the system.
Inactivity goes up because more vehicles are available, but also because there is sufficient extra capacity to wait for the next load at the current location rather than to relocate.

Figure~\ref{fig:maxk_payback} analyzes how long it would take to make back the purchasing costs of the autonomous trucks with the savings obtained by the ATHN.
The plot presents results for autonomous trucks that cost \$150k to \$250k each, based on the \$200k estimated by \citep{ArizonaBankTrust-AutonomousTruckCost}.
For the baseline scenario with $K=100$ vehicles, the figure shows a payback period of less than 1.5 years, demonstrating a rapid return on investment for all parameters.
Again it can be seen that the first vehicles are the most profitable, and a system with only 10 trucks would pay back the trucks within half a year.

\begin{figure}[!t]
	\centering	\includegraphics[width=0.6\textwidth]{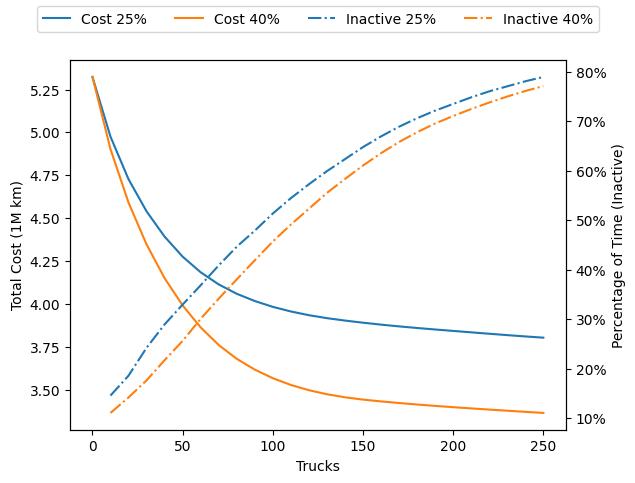}
	\caption{Total Cost and Autonomous Truck Inactivity}
	\label{fig:maxk_inactive}
\end{figure}

\begin{figure}[!t]
	\centering	\includegraphics[width=0.90\textwidth]{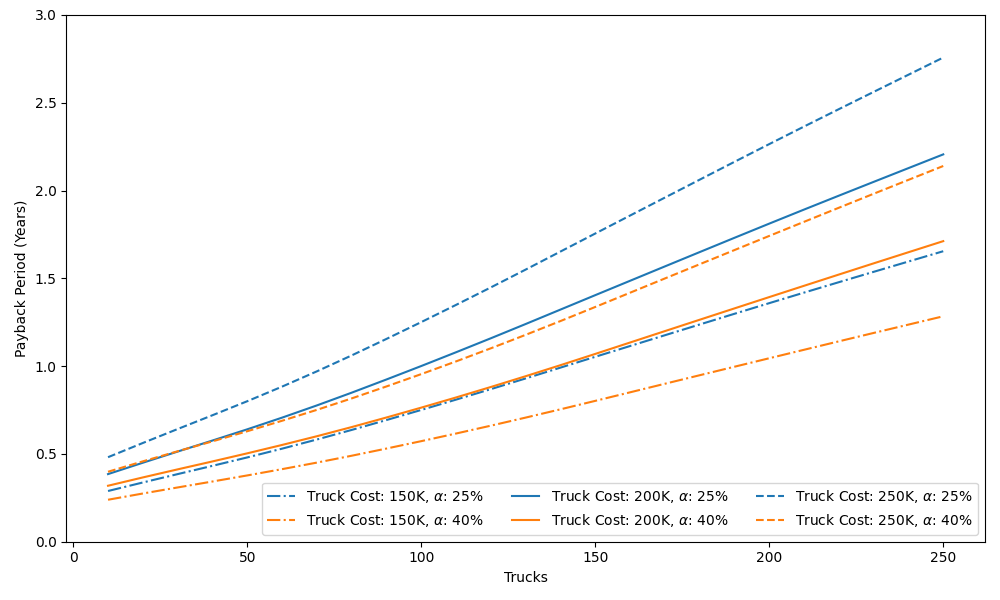}
	\caption{Payback Period for Capital Investments in Autonomous Trucks}
	\label{fig:maxk_payback}
\end{figure}

The number of autonomous trucks in the system also affects the routes that are created.
Figures~\ref{fig:route_10} and~\ref{fig:route_150} present routes taken from the $K = 10$ and $K = 150$ vehicle solutions, respectively.
Blue arrows represent loaded trucks, and red arrows correspond to driving empty.
For the $K = 10$ case, nine of the ten routes are clustered in the East (Figures~\ref{fig:east_route_1} and~\ref{fig:east_route_2}), and one route also serves loads in the West (Figure~\ref{fig:west_route}).
It is clear that the optimization uses the limited number of trucks to serve the most dense region, as shown in Figure~\ref{fig:athn_design_100}.
When the number of trucks increases, Figure~\ref{fig:route_150} shows that routes start to cover all regions of the US.
The Ryder case study therefore suggests a roll-out strategy that starts in the East and expands from there.

\begin{figure}[!tp]
	\centering
	\begin{subfigure}{0.3275\textwidth}
		\centering
		\includegraphics[width=\textwidth]{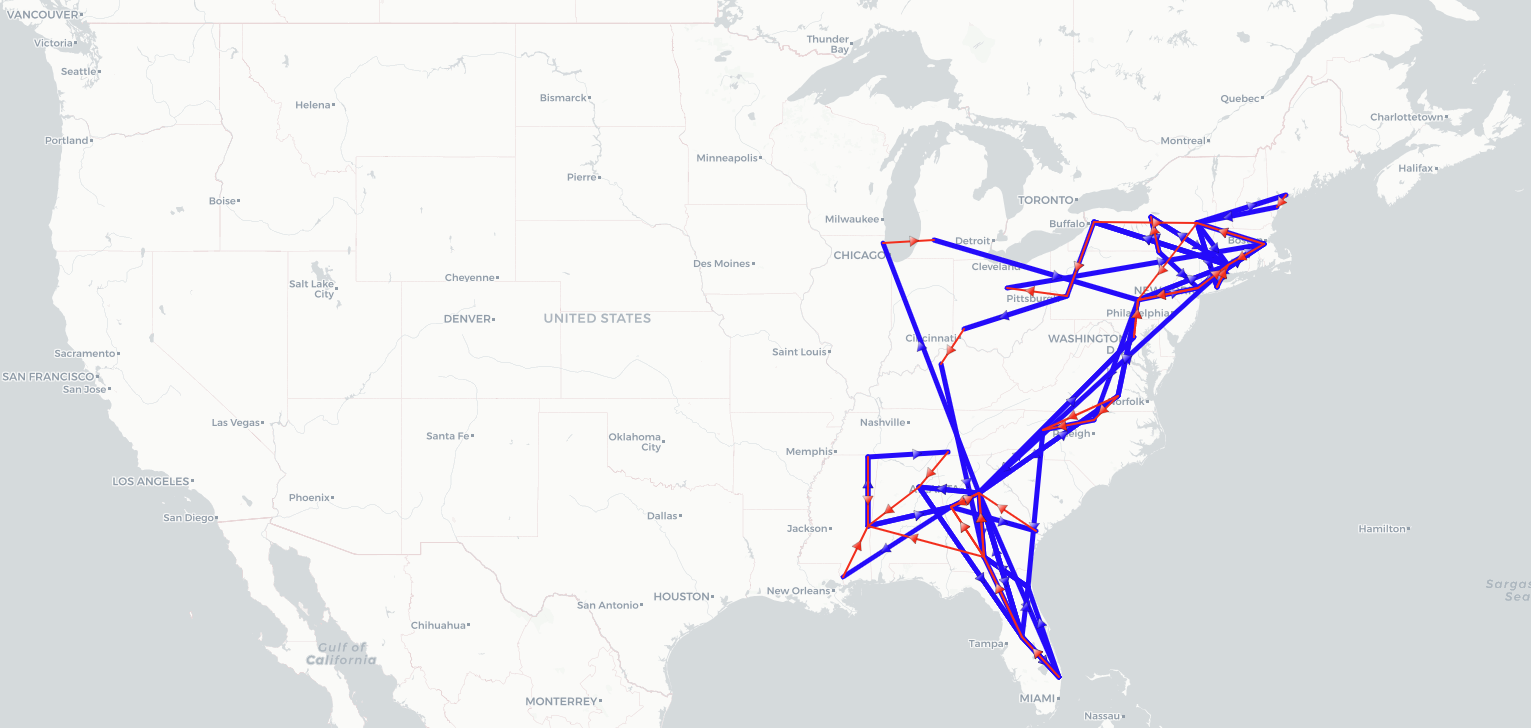}
		\caption{East-focused Route 1}
		\label{fig:east_route_1}
	\end{subfigure}
	\begin{subfigure}{0.3275\textwidth}
		\centering
		\includegraphics[width=\textwidth]{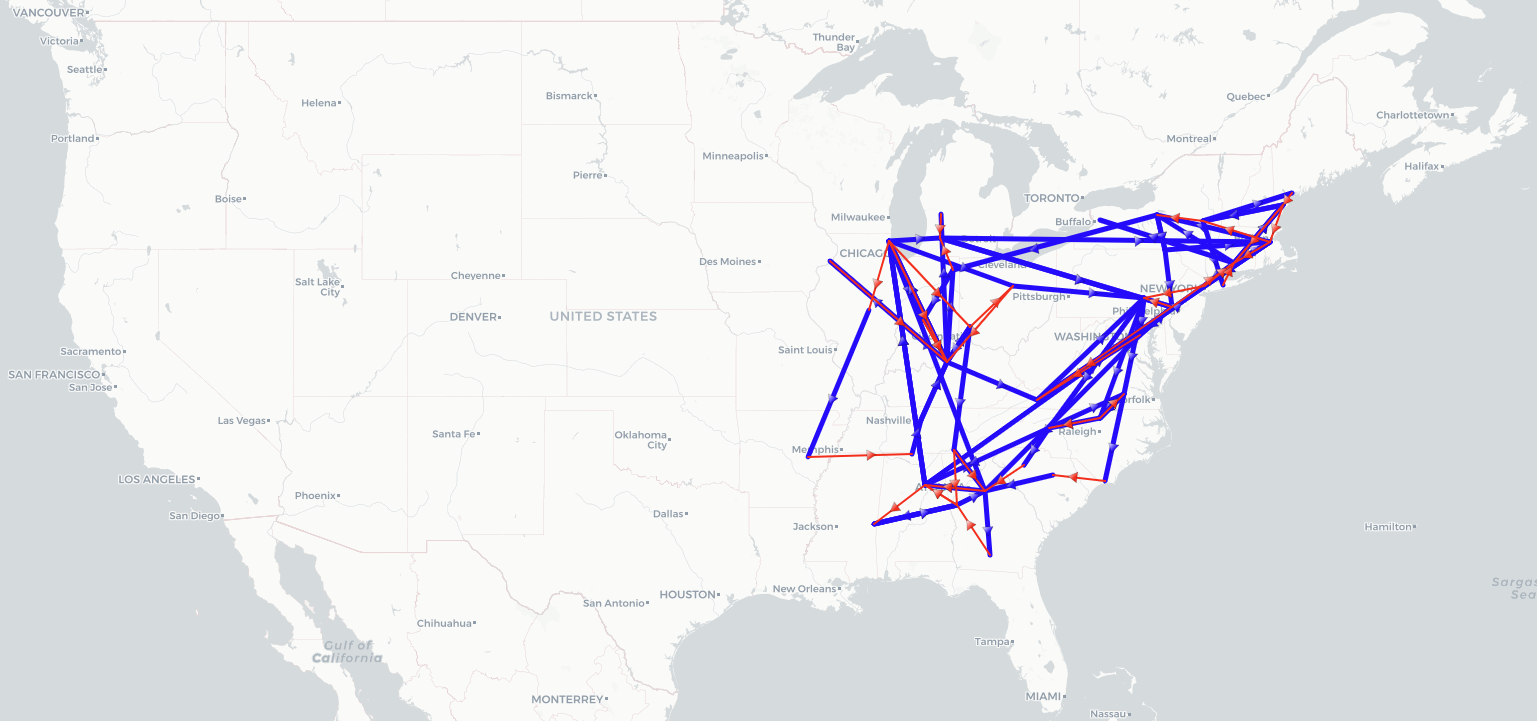}
		\caption{East-focused Route 2}
		\label{fig:east_route_2}
	\end{subfigure}
	\begin{subfigure}{0.3275\textwidth}
		\centering
		\includegraphics[width=\textwidth]{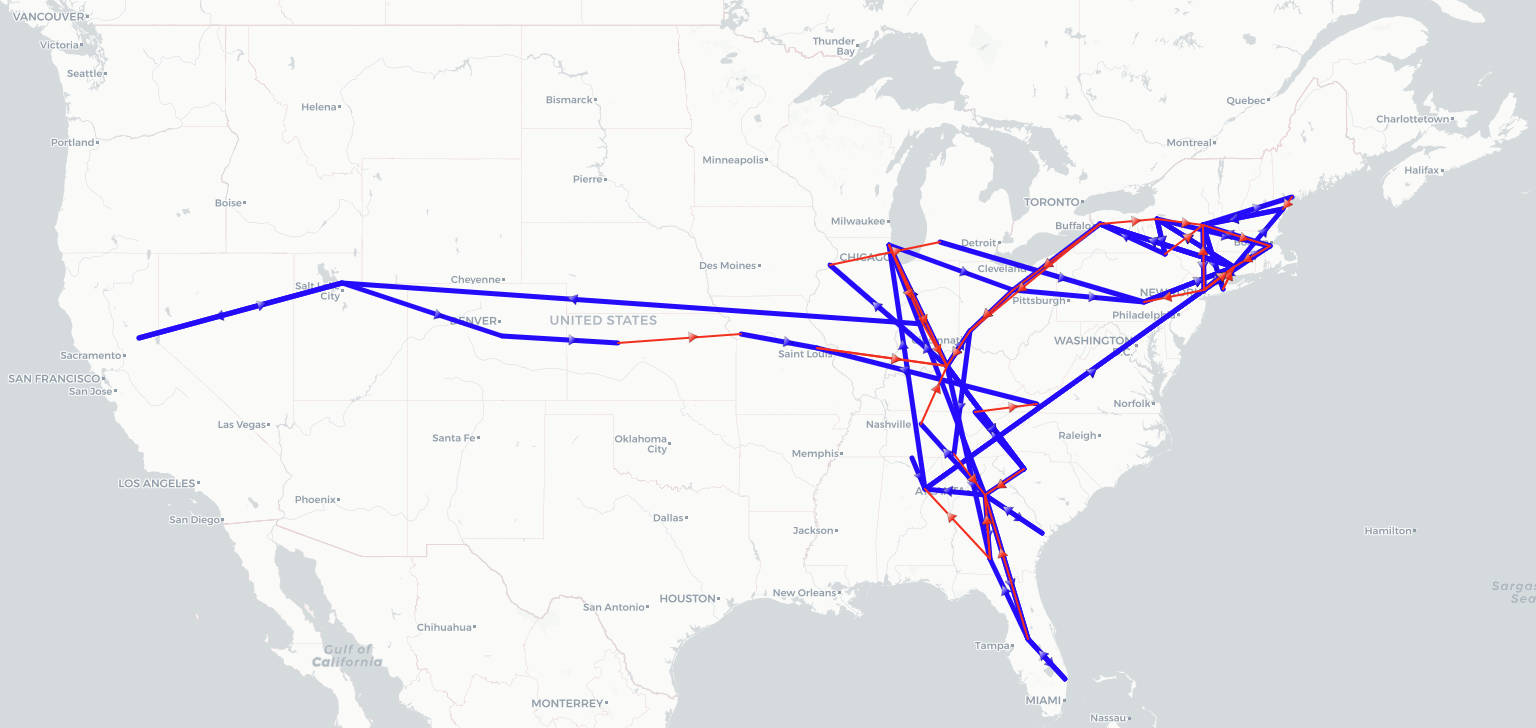}
		\caption{West-focused Route}
		\label{fig:west_route}
	\end{subfigure}
	\caption{Example Routes for $K = 10$}
	\label{fig:route_10}
\end{figure}

\begin{figure}[!tp]
	\centering
	\begin{subfigure}{0.3275\textwidth}
		\centering
		\includegraphics[width=\textwidth]{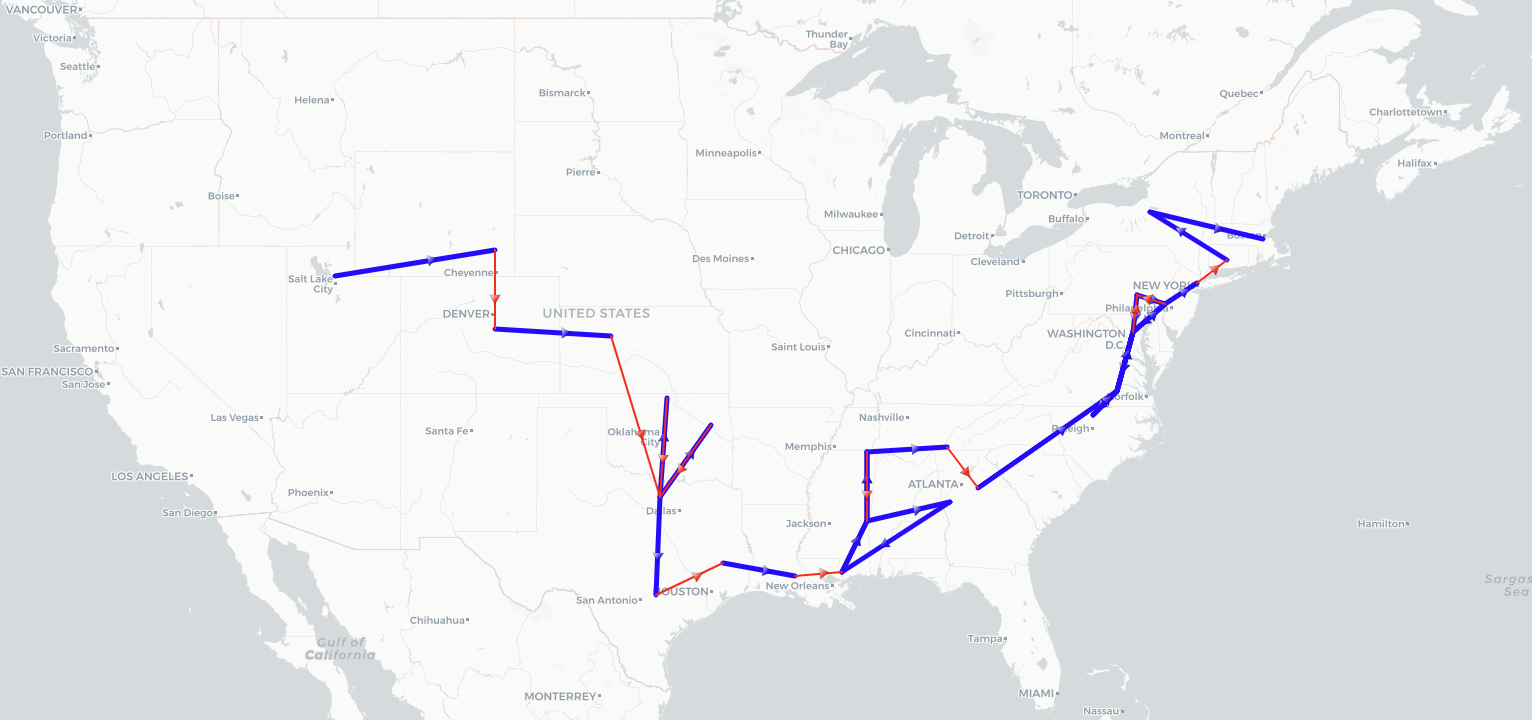}
		\caption{Route 1}
		\label{fig:route_1}
	\end{subfigure}
	\begin{subfigure}{0.3275\textwidth}
		\centering
		\includegraphics[width=\textwidth]{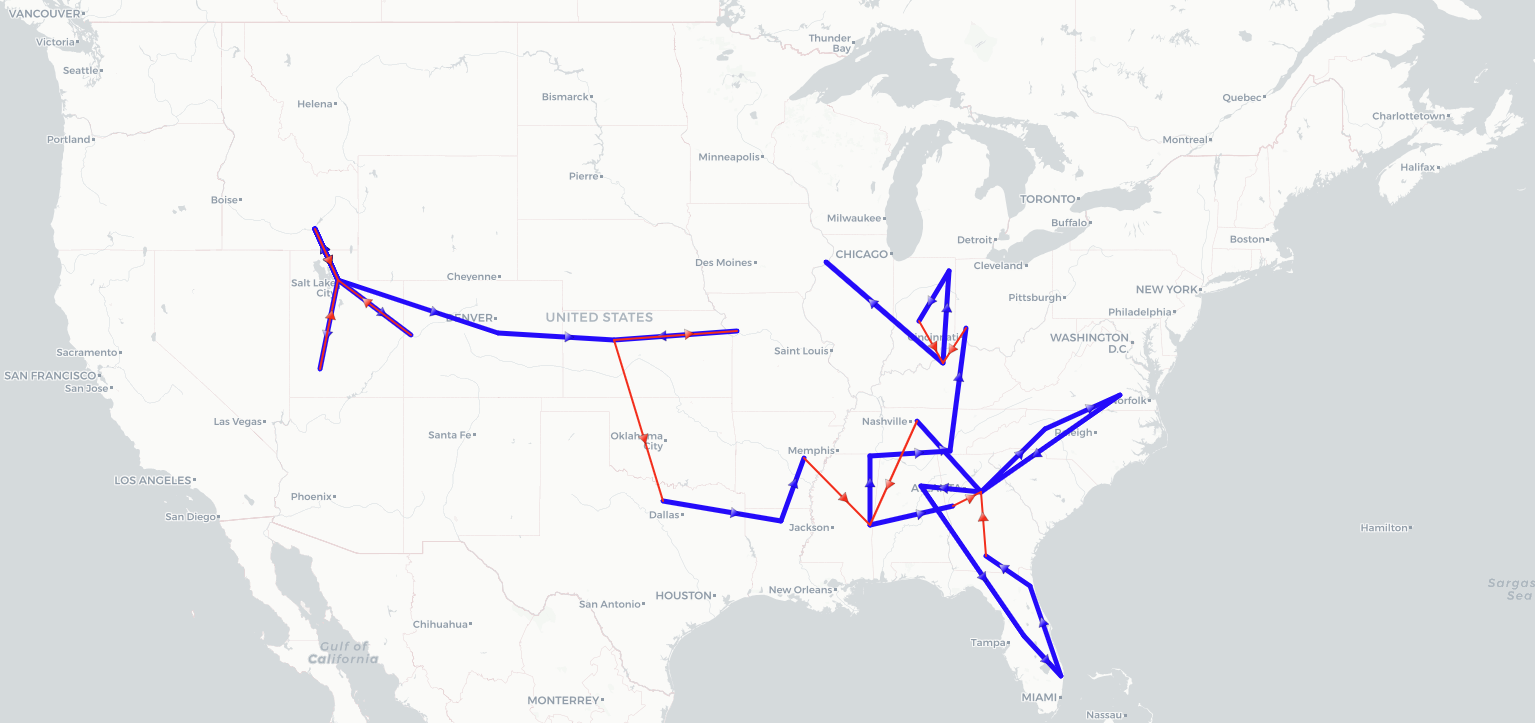}
		\caption{Route 2}
		\label{fig:route_2}
	\end{subfigure}
	\begin{subfigure}{0.3275\textwidth}
		\centering
		\includegraphics[width=\textwidth]{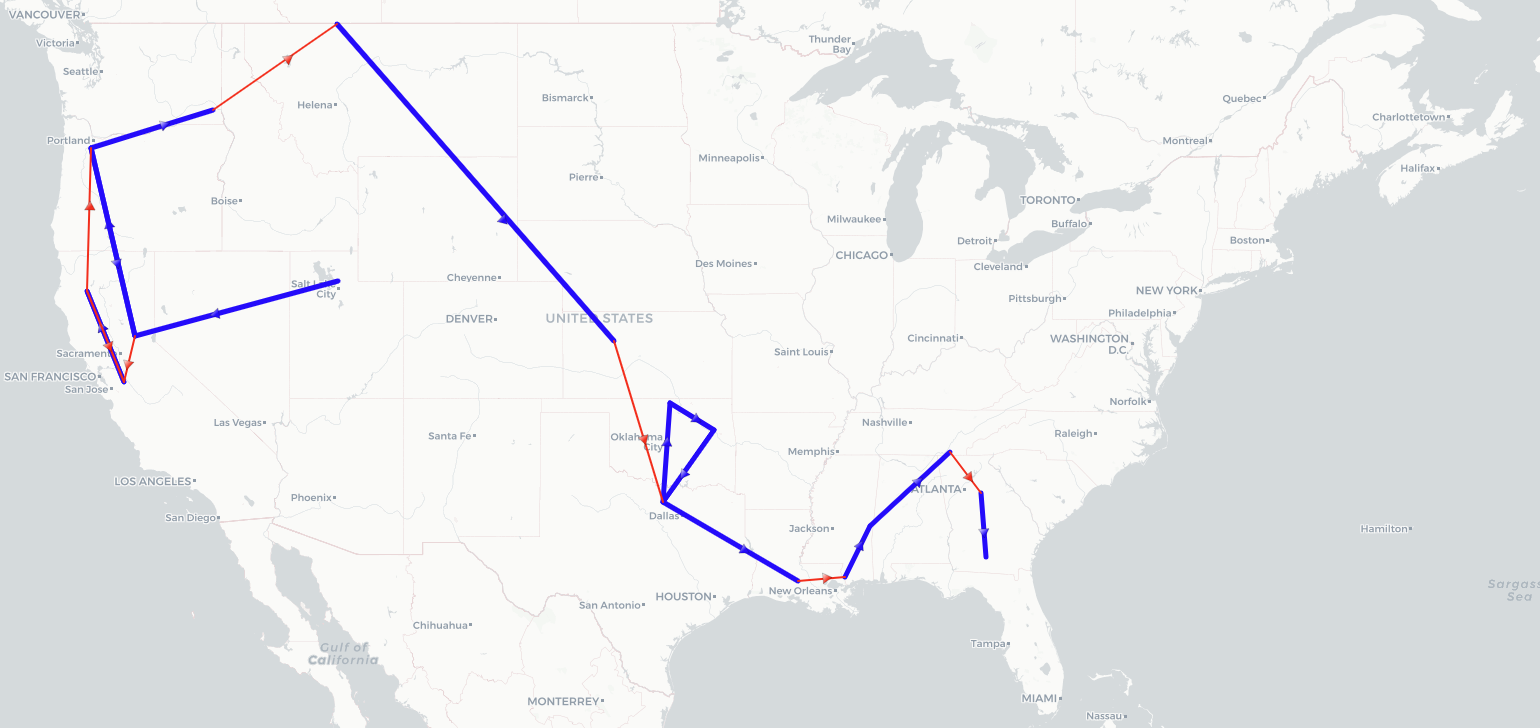}
		\caption{Route 3}
		\label{fig:route_3}
	\end{subfigure}
	\caption{Example Routes for $K = 150$}
	\label{fig:route_150}
\end{figure}


\subsection{Regional and Temporal Decomposition}
In practice it may be necessary to plan ATHN operations for a specific region or for a specific time horizon.
Examples are if states do not share order information, or if loads are only announced a week in advance.
Decomposing the problem is also a strategy to speed up the optimization.

\paragraph{Regional Decomposition}
The US is comprised of smaller areas where the ATHN framework can be applied separately according to Section~\ref{sec:reg_temp}.
This study focuses on the impact of regional decomposition on the ATHN problem. 
Three regional decomposition schemes: 1. Region, 2. Division, and 3. State, adopted from the \citet{CensusBureau2010-CensusRegionsDivisions}, are used to solve the base cases with $K = 100$.

\begin{figure}[!tp]
	\centering
	\begin{subfigure}{0.47\textwidth}
		\centering
		\includegraphics[width=\textwidth]{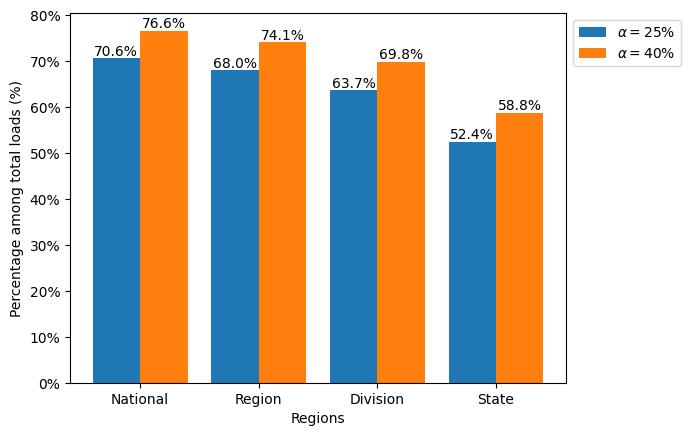}
		\caption{Loads}
		\label{fig:regional_load}
	\end{subfigure}
	\begin{subfigure}{0.47\textwidth}
		\centering
		\includegraphics[width=\textwidth]{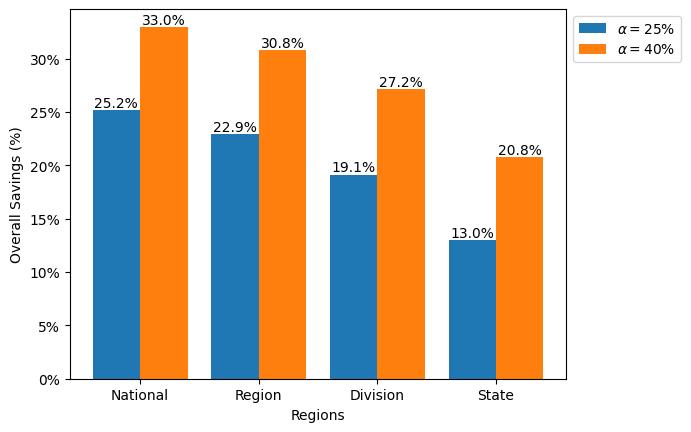}
		\caption{Overall Savings}
		\label{fig:regional_save_overall}
	\end{subfigure}
	\caption{Statistics for Regional Decomposition}
	\label{fig:regional}
\end{figure}

Figure~\ref{fig:regional} presents the summary of the regional decomposition.
As expected, a finer regional decomposition decreases the number of loads served by the ATHN and lowers the savings.
But surprisingly, most of the benefits are still obtained when ATHN operations are planned on the region or division level.
Planning at the state level is significantly more expensive, although ATHN still outperforms the current system by 13\% to 21\%.
This indicates that a successful system is likely to require collaboration between states.
Figure~\ref{fig:allocation} presents the allocation of autonomous trucks to regions or divisions.
Recall that this is optimized by the model as a byproduct of the regional decomposition (Section~\ref{sec:reg_temp}).
Both plots indicate that a majority of the trucks are allocated to the South and in the Northeast, in line with the data presented in Figure~\ref{fig:athn_design_100}.
Figure~\ref{fig:REGIONAL_routes} presents an example route for the Northeast region for $\alpha = 25\%$.
Note that vehicles are allowed to leave the region, but have to return empty before they can pick up the next load.
This corresponds to a situation where regions have joint infrastructure but do not communicate about orders.

\begin{figure}[!tp]
	\centering
	\begin{subfigure}{0.47\textwidth}
		\centering
		\includegraphics[width=\textwidth]{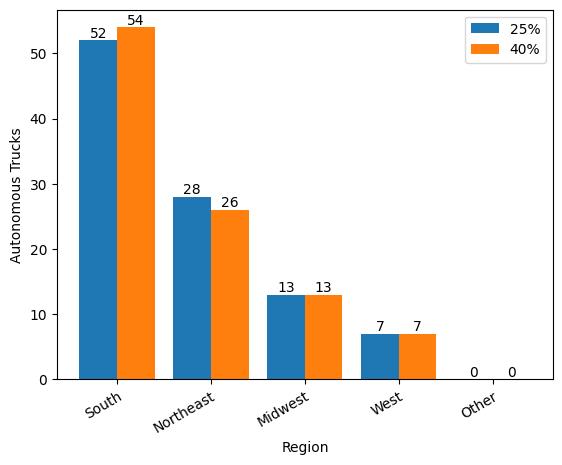}
	\end{subfigure}
	\begin{subfigure}{0.47\textwidth}
		\centering
		\includegraphics[width=\textwidth]{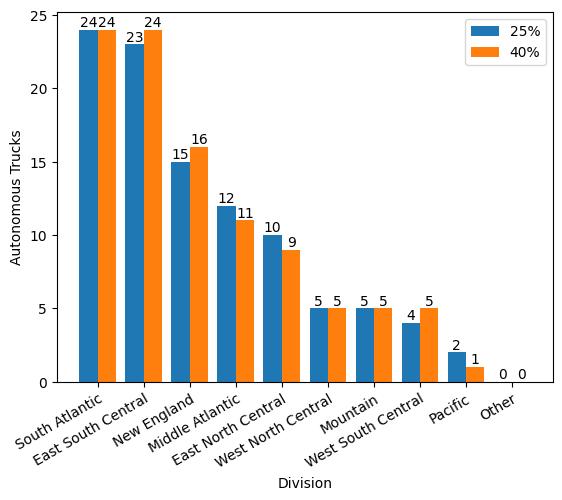}
	\end{subfigure}
	\caption{Vehicle Allocation per Area}
	\label{fig:allocation}
\end{figure}

\begin{figure}[!tp]
	\centering
	\includegraphics[width=0.6\textwidth]{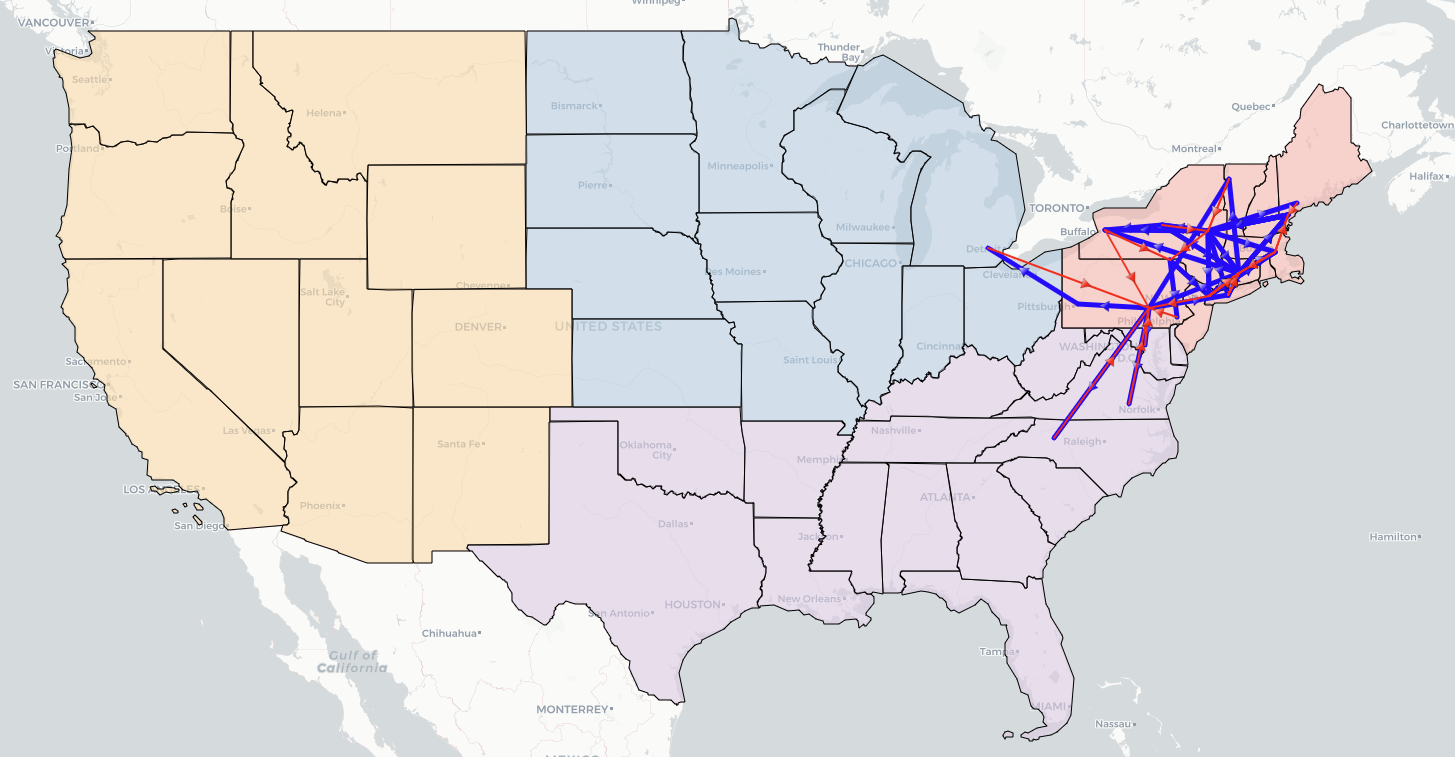}
	\caption{Example Route Northeast Region for $\alpha = 25\%$}
	\label{fig:REGIONAL_routes}
\end{figure}

\paragraph{Temporal Decomposition}
Figure~\ref{fig:temporal} summarizes the results for decomposing the optimization problem into smaller time periods that are solved on a rolling horizon as described in Section~\ref{sec:reg_temp}.
Surprisingly, Figure~\ref{fig:temporal_cost} shows that the ATHN can be operated on a weekly basis at only a minimal increase in cost.
This indicates that a week is sufficiently long to deal with the vehicle imbalances created during the previous period, without the need for explicit rebalancing.
Figure~\ref{fig:temporal_load} shows that essentially the same number of loads are served whether the operations are planned on a monthly, biweekly, or weekly basis.
This has the potential to greatly simplify planning, as loads do not have to be announced too far in advance.
When the planning period is reduced to a single day, it can be seen that the costs do increase.
More noticeably, the number of served loads drops significantly.
This suggests that the optimization model is still able to perform the most profitable tasks, but is unable to exploit all the opportunities because the autonomous trucks are not in the right place at the beginning of the day.

\begin{figure}[!tp]
	\centering
	\begin{subfigure}{0.47\textwidth}
		\centering
		\includegraphics[width=\textwidth]{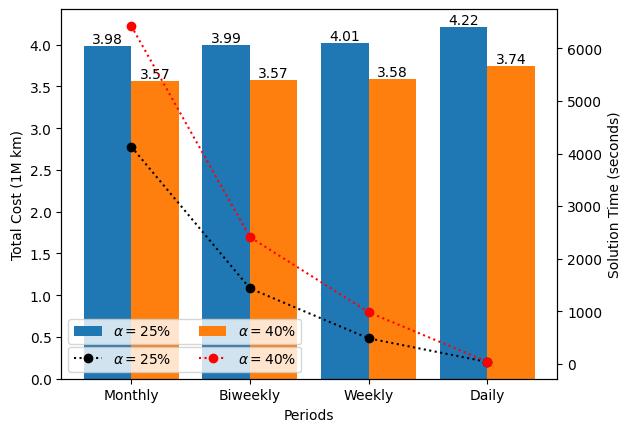}
		\caption{Total Cost and Computational Time}
		\label{fig:temporal_cost}
	\end{subfigure}
	\begin{subfigure}{0.47\textwidth}
		\centering
		\includegraphics[width=\textwidth]{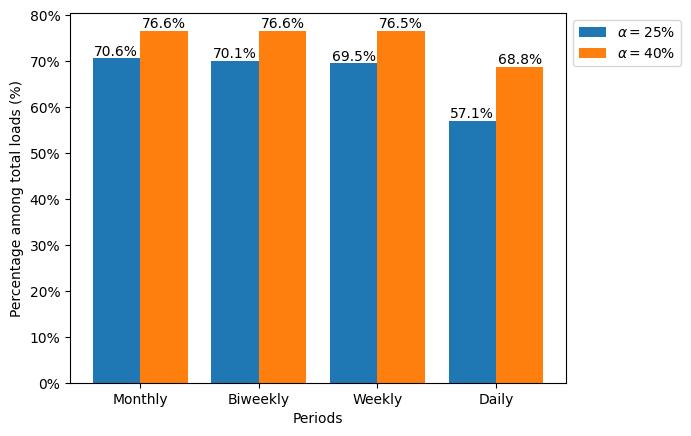}
		\caption{Loads}
		\label{fig:temporal_load}
	\end{subfigure}
	\caption{Statistics for Temporal Decomposition}
	\label{fig:temporal}
\end{figure}

In terms of solution time, there is also a significant advantage to planning on a weekly basis.
Figure~\ref{fig:temporal_cost} includes two line plots that indicate the \emph{total} solving time over the full horizon.
For the $\alpha=25\%$ case, for example, it shows that planning biweekly instead of monthly only requires 35\% of the solving time, while planning weekly instead of monthly only requires 12\% of the solving time.
It is concluded for the case study that planning weekly is a practical and fast alternative to planning monthly.
Although the quality is worse, planning daily is extremely fast and may be useful to adjust schedules when unforeseen events happen.


\subsection{Loading and Unloading Time}

The baseline assumes that loading or unloading an autonomous truck takes 30 minutes.
However, there remains uncertainty in the time this will take in practice.
For example, operators may choose to perform more extensive inspections that increase the loading and unloading time.
Figure~\ref{fig:loadunload} presents results for the $K=100$ case when the loading/unloading time $S$ is varied from zero up to two hours.
All instance were solved to optimality, except for the case $\alpha = 40\%$ and $S = 2$ hours which remained at a small gap of less than 0.01\%.
It can be seen that increasing the loading/unloading time reduces the amount of loads that are served autonomously.
At the same time, the overall savings go down, but not by as much.
This indicates that when loading/unloading time increases, the optimization model becomes more selective about which loads to serve autonomously and which loads to serve with direct trips.
This allows for accommodating a significant increase in loading/unloading time without substantial cost increases.

\begin{figure}[!tp]
	\centering
	\begin{subfigure}{0.47\textwidth}
		\centering
		\includegraphics[width=\textwidth]{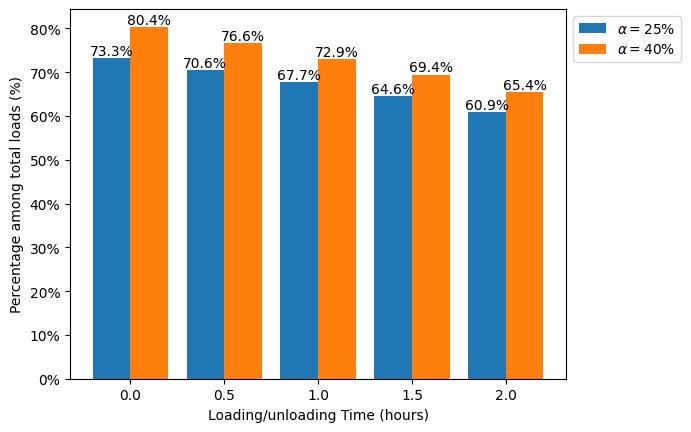}
		\caption{Loads}
		\label{fig:loadunload_load}
	\end{subfigure}
	\begin{subfigure}{0.47\textwidth}
		\centering
		\includegraphics[width=\textwidth]{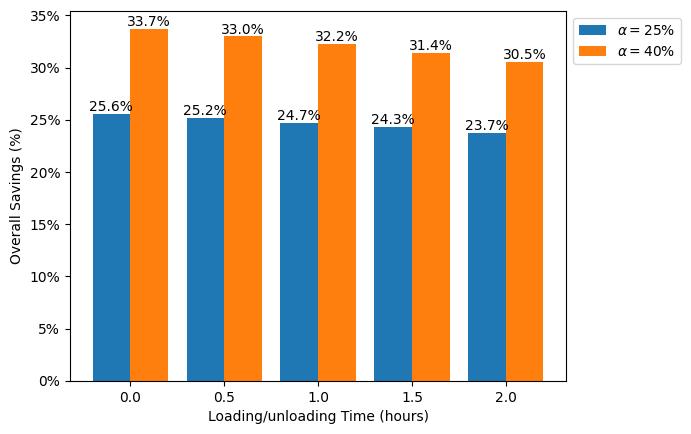}
		\caption{Overall Savings}
		\label{fig:loadunload_save_overall}
	\end{subfigure}
	\caption{Statistics for Varying Loading/Unloading Time}
	\label{fig:loadunload}
\end{figure}

\subsection{Network Size}

Finally, it is explored how the number of transfer hubs affects the ATHN.
Figure~\ref{fig:hubs} summarizes the results.
Figure~\ref{fig:hubs_load} shows that as the number of hubs increases, more loads will be served autonomously.
This effect is due to the shorter first and last mile distances that make it more attractive to use the ATHN.
Figure~\ref{fig:hubs_save_overall} shows that the ATHN remains at least 19.7\% cheaper than the current system, even when the number of hubs is reduced to 50.
This is an important fact for practical implementation: it is not necessary to deploy the full network at once to obtain substantial savings.
For the case study, additional benefits can be obtained by further extending to 200 hubs.
In this case the benefits mainly derive from making the first and last miles shorter, rather than capturing additional loads.

\begin{figure}[!tp]
	\centering
	\begin{subfigure}{0.47\textwidth}
		\centering
		\includegraphics[width=\textwidth]{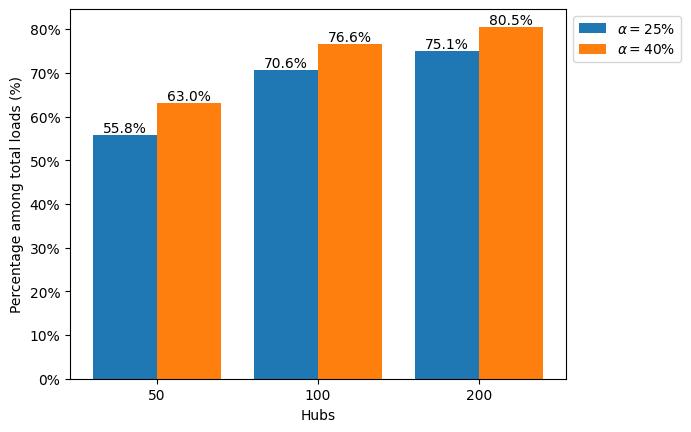}
		\caption{Loads}
		\label{fig:hubs_load}
	\end{subfigure}
	\begin{subfigure}{0.47\textwidth}
		\centering
		\includegraphics[width=\textwidth]{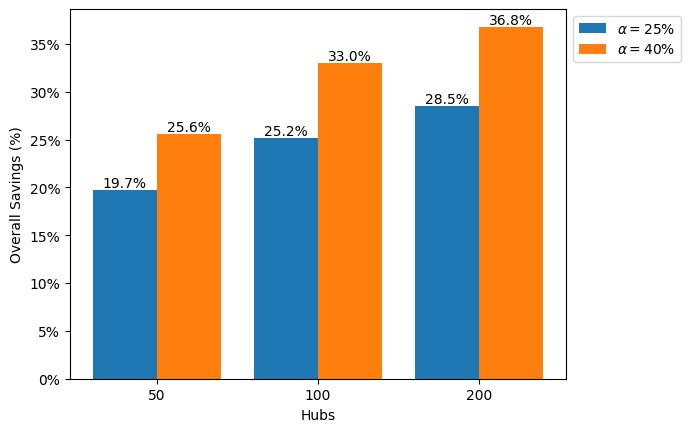}
		\caption{Overall Savings}
		\label{fig:hubs_save_overall}
	\end{subfigure}
	\caption{Statistics for Varying Number of Hubs}
	\label{fig:hubs}
\end{figure}

\section{Conclusion}
\label{sec:conclusion}

Autonomous trucks are expected to fundamentally transform the freight transportation industry.
In particular, Autonomous Transfer Hub Networks (ATHNs), which combine autonomous trucks on middle miles with human-driven trucks on the first and last miles, are seen as the most likely deployment pathway for this technology. 
This paper presented a framework to optimize ATHN operations based on a flow-based optimization model.
The optimization model exploits the problem structure and can be solved by blackbox solvers in a matter of hours.
Tools were also provided to perform regional and temporal decomposition in the same framework.

A realistic case study was performed to demonstrate the capability of the optimization model to handle large-scale systems.
The case study is based on realistic data provided by Ryder System, Inc., that spans all of the United States over a four-week horizon.
Computational results were presented for a system of 100 hubs, 6842 loads, and up to 250 vehicles.
The results show that a blackbox solver can solve these instances to within 0.18\% of optimality within three hours of computation time.

The case study itself provided many insights into the impact of self-driving trucks on freight transportation.
The baseline results indicate cost savings in the range of 20\% to 37\% for the most challenging orders that are currently served by direct trips.
This corresponds to upward of \$16.9M per year on the case study.
A detailed sensitivity analysis was provided to study various changes to the system.
It was shown that increasing the pickup-time flexibility leads to additional potential savings that cannot yet be exploited by the current algorithm, which motivates future work.
Using a MIP start significantly improved the performance for these difficult instances.
The analysis of the number of trucks demonstrates the trade-off between cost reduction and truck inactivity.
It also shows the recommended path of implementation for the case study, starting with trucks that exclusively serve the busy East and West, and moving to routes that cover the whole country as the number of trucks increases.

Decomposing the optimization problems into regions and into smaller time periods led to some interesting practical insights.
For the case study, it was shown that planning on the regional level maintains many of the benefits of planning on the national level.
Similarly, it was found that planning one week ahead on a rolling horizon leads to surprisingly good results and also significantly reduced solving time by 88\%.
Daily planning is extremely fast and may be useful to adjust schedules when unforeseen events happen.
The analysis of the loading/unloading time showed that long loading/unloading times of up to two hours (e.g., to perform inspections) can be accommodated with minimal impact on the system's cost-effectiveness.
And finally, varying the size of the network demonstrated that the ATHN can still be run efficiently when the number of hubs is reduced from 100 to 50.

In conclusion, it is found that the optimization framework for ATHN is an effective tool to study large full truckload systems and gain insights from real data.
Interesting directions for future work include studying the effect of ATHN on human labor beyond the drivers, e.g., support staff at the transfer hubs.
A first step in that direction has been taken by \citep{LeeEtAl2023-ConstraintProgrammingImprove}, which presents a constraint programming method to optimize the size of each hub.
Another future direction could be to incorporate the possibility of platooning into the model, to save additional costs.
Finally, it would be interesting to extend the framework to less-than-truckload transportation.

\subsubsection*{Acknowledgments}
This research was partly funded through a gift from Ryder and partly supported by the NSF AI Institute for Advances in Optimization (Award 2112533). Special thanks to the Ryder team for their invaluable support, expertise, and insights.

\bibliographystyle{trc}
\bibliography{references}

\end{document}